\documentclass[10pt,letterpaper]{article}

\usepackage[margin=1in]{geometry}
\usepackage[english]{babel}
\usepackage[T1]{fontenc}
\usepackage[utf8]{inputenc}
\usepackage{lmodern}
\usepackage{microtype}
\usepackage{ragged2e}
\usepackage{xparse}

\usepackage{amsmath,amssymb,amsfonts,amsthm,mathtools}
\usepackage{mathrsfs}
\newtheorem{hypothesis}{Hypothesis}
\newtheorem*{hypothesis*}{Hypothesis}
\usepackage{bm}
\usepackage{dsfont}
\usepackage[normalem]{ulem}
\usepackage{graphicx}
\usepackage{xcolor}
\definecolor{marron}{rgb}{0.5,0.25,0.0}
\usepackage{verbatim}
\usepackage{tikz}
\usepackage{subcaption}
\usetikzlibrary{arrows.meta,positioning}
\usepackage{enumitem}
\setlist[enumerate]{
    topsep=2pt,
    itemsep=1pt,
    parsep=0pt,
    partopsep=0pt,
    leftmargin=1.8em,
    font=\textbf
}

\setlist[description]{
    topsep=2pt,
    itemsep=1pt,
    parsep=0pt,
    partopsep=0pt,
    leftmargin=1.8em,
    font=\textbf
}

\newtheoremstyle{compactplain}
  {4pt}{4pt}
  {\itshape}
  {}
  {\bfseries}
  {.}
  {0.5em}
  {}

\newtheoremstyle{compactdefinition}
  {4pt}{4pt}
  {\normalfont}
  {}
  {\bfseries}
  {.}
  {0.5em}
  {}

\theoremstyle{compactplain}
\newtheorem{theorem}{Theorem}
\newtheorem*{theorem*}{Theorem}
\newtheorem{lemma}[theorem]{Lemma}
\newtheorem{corollary}[theorem]{Corollary}
\newtheorem{proposition}[theorem]{Proposition}

\theoremstyle{compactdefinition}

\newtheorem{exercise}[theorem]{Exercise}
\newtheorem{example}[theorem]{Example}
\newtheorem{remark}[theorem]{Remark}

\DeclareMathOperator{\supp}{supp}

\renewcommand{\Pr}{\mathbb P}

\newcommand{\RR}{\mathbb R}
\newcommand{\Rp}{\mathbb R_+}

\newcommand{\rar}{\rightarrow}

\newcommand{\bit}{\begin{itemize}}
\newcommand{\eit}{\end{itemize}}
\newcommand{\ben}{\begin{enumerate}}
\newcommand{\een}{\end{enumerate}}
\newcommand{\bdes}{\begin{description}}
\newcommand{\edes}{\end{description}}
\newtheorem*{theoremA}{Theorem (A)}

\makeatletter

\renewcommand{\bit}{\@ifnextchar[{\bit@opt}{\begin{itemize}}}
\def\bit@opt[#1]{\begin{itemize}[#1]}
\renewcommand{\eit}{\end{itemize}}

\renewcommand{\ben}{\@ifnextchar[{\ben@opt}{\begin{enumerate}}}
\def\ben@opt[#1]{\begin{enumerate}[#1]}
\renewcommand{\een}{\end{enumerate}}

\renewcommand{\bdes}{\@ifnextchar[{\bdes@opt}{\begin{description}}}
\def\bdes@opt[#1]{\begin{description}[#1]}
\renewcommand{\edes}{\end{description}}

\makeatother

\newcommand{\bc}{\begin{center}}
\newcommand{\ec}{\end{center}}

\newcommand{\bthm}{\begin{theorem}}
\newcommand{\ethm}{\end{theorem}}

\newcommand{\bprop}{\begin{proposition}}
\newcommand{\eprop}{\end{proposition}}
\newcommand{\bcor}{\begin{corollary}}
\newcommand{\ecor}{\end{corollary}}
\newcommand{\brem}{\begin{remark}}
\newcommand{\erem}{\end{remark}}
\newcommand{\bex}{\begin{example}}
\newcommand{\eex}{\end{example}}
\newcommand{\bexo}{\begin{exercise}}
\newcommand{\eexo}{\end{exercise}}

\newcommand{\beq}{\begin{equation}}
\newcommand{\eeq}{\end{equation}}
\newcommand{\beqar}{\begin{eqnarray}}
\newcommand{\eeqar}{\end{eqnarray}}
\newcommand{\beqarr}{\begin{eqnarray*}}
\newcommand{\eeqarr}{\end{eqnarray*}}

\newcommand{\iti}{\item[(i)]}
\newcommand{\itii}{\item[(ii)]}
\newcommand{\itiii}{\item[(iii)]}

\renewcommand{\iti}{\item[\textup{(i)}]}
\renewcommand{\itii}{\item[\textup{(ii)}]}
\renewcommand{\itiii}{\item[\textup{(iii)}]}

\usepackage[colorlinks=true,
            linkcolor=blue,
            citecolor=blue,
            urlcolor=blue]{hyperref}

\title{Selection principles for quasi-stationary distributions and reinforcement processes}
\author{
Michel~Bena\"im$^{\dagger}$,
Tresnia~Berah$^{\ddagger}$,
Pierre~Germain$^{\dagger,\S}$
\\[0.3cm]
{\small \textit{$^{\dagger}$Institut de Mathématiques, Université de Neuchâtel, Switzerland. E-mail: michel.benaim@unine.ch}}
\\
{\small \textit{$^{\ddagger}$Department of Mathematics, Imperial College London, United Kingdom. E-mail: tresnia.berah13@ic.ac.uk}}
\\
{\small \textit{$^{\S}$École des Ponts ParisTech, Marne-la-Vallée, France. E-mail: pierre.germain@unine.ch}}
}
\date{\today}

\begin{document}
\maketitle

\begin{abstract}

Let $P$ be a sub-Markov matrix on a finite set $S$, representing the
transition probabilities of a Markov chain on \(S\) absorbed at a cemetery
point $\partial\notin S$. We consider a reinforced process
\((X_n,\mu_n)\) defined as follows: $(X_n)$ behaves like a chain with kernel $P$ 
until it dies, and when it dies at time $n$, it is instantaneously
``resurrected'' at a point sampled according to its weighted past occupation
measure
$$
\mu_n
=
\frac{1}{W_n}
\left(
w_0\mu_0+\sum_{k=1}^n w_k\delta_{X_k}
\right),
\qquad
W_n=\sum_{k=0}^n w_k,
$$
where the positive weights $w_k$ satisfy certain technical assumptions. A typical example is given by $w_0=1$ and $w_k=k^q$ for $k\geq1$, with $q\geq-1$.
When $P$ is irreducible, the behaviour of $(\mu_n)$ is well understood \cite{AFP}, \cite{bansaye2022non}: it converges almost surely toward the
unique quasi-stationary distribution (QSD) of $P$. The purpose of this paper is to investigate the general situation where $P$ is not irreducible. Under
generic assumptions on $P$, there are finitely many QSDs. We prove that the asymptotic selection depends on the summability of the inverse cumulative weights $1/W_n$. If
$$
\sum_{n\geq 0}\frac1{W_n}=\infty,
$$
then $(\mu_n)$ almost surely converges toward the QSD associated with the largest Perron value. If instead
$$
\sum_{n\geq0}\frac1{W_n}<\infty,
$$
then $(\mu_n)$ converges almost surely toward a QSD, and each QSD is selected with positive probability. In particular, for polynomial weights $w_0=1$ and $w_k=k^q$, $k\geq1$, this gives almost
sure selection of the QSD with largest Perron value for $-1\leq q\leq 0$, whereas each quasi-stationary distribution is selected with positive probability for $q>0$.
\end{abstract}

% \tableofcontents

\section{Introduction}

Let \(S\) be a finite state space and let \(P=(P(x,y))_{x,y\in S}\) be a sub-Markovian kernel on \(S\), namely
\begin{align}\label{def:sub-markov-matrix}
    P(x,y)\geq0,
\qquad
0<\sum_{y\in S}P(x,y)\leq1,
\qquad x,y\in S.
\end{align}

Such a kernel naturally defines a killed Markov chain by adjoining a cemetery
state \(\partial\notin S\) and considering the Markov kernel \(\widehat P\) on
\(S\cup\{\partial\}\) given by
\[
\widehat P(x,y)=P(x,y),
\qquad
\widehat P(x,\partial)=1-\sum_{y\in S}P(x,y),
\qquad x,y\in S,
\]
together with
\[
\widehat P(\partial,\partial)=1,
\qquad
\widehat P(\partial,y)=0,
\qquad y\in S.
\]
We denote by \((\widehat X_n)_{n\geq0}\) the Markov chain with kernel
\(\widehat P\). Its law and expectation are denoted by
\(\widehat{\mathbb P}_x\) and \(\widehat{\mathbb E}_x\) when
\(\widehat X_0=x\), and by \(\widehat{\mathbb P}_\nu\) and
\(\widehat{\mathbb E}_\nu\) when \(\widehat X_0\sim\nu\).

A central object in the study of killed Markov processes is the notion of quasi-stationary distribution. We denote by \(\mathcal P(S)\) the set of probability measures on \(S\). A probability measure \(\nu\in\mathcal P(S)\) is called a quasi-stationary distribution (QSD) if, for every
\(n\geq0\),
\[
\nu(\cdot)
=
\widehat{\mathbb P}_\nu
\left(
\widehat X_n\in\cdot
\,\middle|\,
n<\tau_\partial
\right),
\]
where
\[
\tau_\partial:=\inf\{n\geq0:\widehat X_n=\partial\}
\]
denotes the absorption time.
QSDs play the role of invariant measures for non-conservative dynamics. Under Hypothesis~\ref{hyp:main} below, absorption occurs almost surely, but
the law conditioned on survival may exhibit a non-trivial equilibrium
behaviour.

\medskip{}

The theory of QSDs for irreducible sub-Markovian kernels is by now well understood. In particular, irreducibility implies uniqueness of a QSD. In contrast, the non-irreducible setting is substantially richer and more delicate. Multiple QSDs may coexist, reflecting the possible presence of several communicating classes. In the present work, we focus precisely on this reducible framework. Under an additional structural assumption introduced later, the set of QSDs is finite. A natural and challenging question is then to design stochastic algorithms capable of sampling or selecting these different QSDs.

The purpose of this article is to study such a sampling procedure based on reinforced stochastic dynamics. Our approach fits within the general framework of stochastic approximation and self-interacting processes. In order to define the process we wish to study, let us first recall the definition of the resurrected Markov kernel given a fixed measure $\mu$.

For \(x\in S\), let
\[
q(x):=1-\sum_{z\in S}P(x,z)
\]
denote the killing probability at state \(x\). For a probability measure \(\mu\) on \(S\), we define the reinforced transition kernel by
\begin{equation}\label{eq:defnoyauressamp}
K_\mu(x,y):=P(x,y)+q(x)\mu(y),
\qquad x,y\in S.
\end{equation}

The reinforced process is constructed recursively as follows. Its definition only depends on the state space $S$, the given sub-Markovian kernel $P$, the initial
condition \((X_0,\mu_0)\) and a deterministic step-size sequence $(\gamma_n)_{n\geq1}\in(0,1)^{\mathbb{N}^*}$.

\medskip{}

We start from an initial condition \((X_0,\mu_0)\), possibly random, such that
\[
(X_0,\mu_0)\in S\times\mathcal P(S)
\qquad\text{almost surely}.
\]
The reinforced process is defined recursively with respect to the filtration
\[
\mathcal F_n=\sigma(\mu_0,X_0,\ldots,X_n),
\]
through the iteration
\begin{equation} \label{eq:defprocessus}
    \left\{
\begin{aligned}
\mathbb P(X_{n+1}=y\,|\,\mathcal F_n)
&=
K_{\mu_n}(X_n,y),
\qquad y\in S,\\
\mu_{n+1}
&=
(1-\gamma_{n+1})\mu_n+\gamma_{n+1}\delta_{X_{n+1}}.
\end{aligned}
\right.
\end{equation}
Throughout the article, \(\mathbb P\) and \(\mathbb E\) denote probability and
expectation for the reinforced process \((X_n,\mu_n)\) defined by \eqref{eq:defprocessus} and the initial condition  $(X_0,\mu_0)$. The notation
\(\widehat{\mathbb P}\) and \(\widehat{\mathbb E}\) is reserved for the killed
chain with kernel \(\widehat P\).

In the classical stochastic approximation setting where $\gamma_n=\frac1{n+1}$ and where $\mu_0=\delta_{X_0}$, $\mu_n$ coincides with the empirical occupation measure of the trajectory up to time $n$

\[\mu_n=\frac{1}{n+1}\sum_{k=0}^n\delta_{X_k}.\]
For a general sequence $(\gamma_n)$, setting 
\[
r_0:= 1, \quad r_n:= \prod_{k=1}^n (1-\gamma_k), \qquad \text{ and } \qquad w_0:=1, \quad w_k:= \frac{\gamma_k}{r_k},
\]
the measure $\mu_n$ should instead be interpreted as a weighted occupation measure
\[
\mu_n
=
\frac1{W_n}
\left(
w_0\mu_0+\sum_{k=1}^n w_k\delta_{X_k}
\right), \qquad  \ W_n := \sum_{k=0}^n w_k.
\]
This decomposition will be formalised in the next section.

\medskip{}

The heuristic interpretation of the reinforced process is the following. Between two killing events, the particle evolves according to the original sub-Markovian dynamics $P$. Whenever the process is killed, instead of being sent to the cemetery state, it is instantaneously resurrected at a position sampled according to its past occupation measure $\mu_n$ (uniformly over the past trajectory in the classical case $\gamma_n=\frac1{n+1}$). Consequently, states that have already been frequently visited are more likely to be selected again after future killing events. The process therefore progressively favours regions of the state space that it has already explored, creating a reinforcement mechanism through its own history. This self-reinforcement phenomenon explains the terminology of reinforced process.

\medskip

Reinforced stochastic approximation algorithms for the simulation of quasi-stationary distributions are now well understood in irreducible settings. In the finite state space case, the reinforced process was introduced in \cite{BC2015}, and was shown to converge toward the unique QSD. An averaging variant which was shown to improve the convergence rate of the algorithm was introduced in \cite{blanchet2016analysis}. Related reinforced mechanisms linked with QSD appeared earlier in the work \cite{AFP}. The stochastic approximation approach was later generalised to compact state spaces in \cite{benaim2018stochastic}, while related diffusion models in bounded domains were investigated in \cite{benaim2021stochastic, panloup2026asymptotically}. A common feature of these works is that the underlying deterministic dynamics admits a unique globally attractive equilibrium, corresponding to the unique QSD of the killed process.

\medskip{}

The reducible setting is substantially more delicate. In this case, several QSDs may coexist and the asymptotic behaviour of the reinforced dynamics becomes much less clear. The structure of QSDs for reducible Markov processes was investigated in \cite{pollett2008quasi}, while a general description of quasi-limiting distributions and spectral dominance phenomena in reducible state spaces was developed in \cite{champagnat2022quasi}. In \cite{benaim2018stochastic}, the authors showed that in this setting trapping phenomena may occur, leading with positive probability to the selection of a QSD that is not associated to the largest Perron eigenvalue. The analysis of this trapping regime is carried out in a two-state example, while the understanding of the general reducible case, as well as the characterization of the weak reinforcement regime, is explicitly formulated as an open problem.

\medskip{}

When several QSDs coexist, a natural question is whether a selection principle holds: does the stochastic algorithm asymptotically select a distinguished QSD? Such phenomena are expected to be strongly related to spectral dominance mechanisms. Results of this type have already been obtained in settings where the multiplicity of QSDs originates from non-compactness of the state space. In particular, reinforced stochastic approximation procedures in non-compact frameworks were studied in \cite{MaillerVillemonais2020}, while related selection results for Fleming--Viot particle systems were obtained in \cite{tough2025selection}. These works strongly suggest the existence of robust spectral selection mechanisms for reinforced quasi-stationary sampling algorithms.

\medskip{}

The purpose of the present work is to investigate the emergence of selection
phenomena for reinforced QSD sampling algorithms in reducible state spaces and
to show that the resulting selection principle is governed by the summability
of the inverse cumulative weights \(1/W_n\). Our main results can be summarized
as follows.

\begin{theoremA}
Under standard hypotheses on the step-size (cf.~Hyp.~\ref{hyp:gamma}) and natural assumptions on $P$ and $(X_0, \mu_0)$ (cf.~Hyp.~\ref{hyp:main}, \ref{hyp:main-two}, \ref{hyp:initial-support}), the following dichotomy holds: 
\begin{enumerate}
    \item[\textup{(i)}]
    If
    \[
    \sum_{n\geq0}\frac{1}{W_n}=\infty,
  \]
    then $(\mu_n)$ converges almost surely towards the QSD associated with the largest Perron value.
    This is Theorem~\ref{th:weak-regime}.

    \item[\textup{(ii)}]
    If
    \[ \sum_{n\geq0}\frac{1}{W_n}<\infty,\]
  then $(\mu_n)$ converges almost surely toward a QSD, and every QSD is selected with positive probability; see Theorems~\ref{th:main} and~\ref{thm:strong-regime}.
\end{enumerate}

\end{theoremA}

The paper is organised as follows.
Section~\ref{sec:notation-results} introduces the notation and states the main results. Section~\ref{sec:general-proof} studies the Green-kernel flow, determines the basins of its equilibria, proves that its only internally
chain-transitive sets are the singleton QSDs, and establishes
Theorem~\ref{th:main}. Section~\ref{sec:selection-weak-regime} proves the weak-reinforcement selection principle, Theorem~\ref{th:weak-regime}.
Section~\ref{sec:strong-regime} proves the strong-reinforcement trapping and selection result, Theorem~\ref{thm:strong-regime}. The genericity of
Hypothesis~\ref{hyp:main-two} is proved in the appendix.

\section{Notations, assumptions and main results}\label{sec:notation-results}

 We denote by \(\RR^S\) the space of real-valued functions on \(S\), and by \((\RR^S)^*\) its dual space, identified with signed measures on \(S\). For every \(f\in\RR^S\) and every \(\mu\in(\RR^S)^*\), we define
\[
Pf(i)=\sum_{j\in S}P(i,j)f(j),
\qquad
\mu P(i)=\sum_{j\in S}\mu(j)P(j,i).
\]
Throughout the paper, \(\mathbf 1\in\RR^S\) denotes the constant function equal to \(1\). Probability measures on \(S\) are identified with the simplex
\[
\Delta:=\left\{\mu\in(\RR^S)^*:\ \mu(i)\geq0,\ \sum_{i\in S}\mu(i)=1\right\}.
\]
The killed kernel \(\widehat P\), the probability of killing \(q(i)\), and the resurrected kernel \(K_\mu\) were introduced in the previous section. For \(\mu\in\Delta\), we define
\[
\supp(\mu):=\{i\in S:\ \mu(i)>0\}.
\]
Given \(i,j\in S\), we say that \(j\) is \emph{accessible} from \(i\), written
\[
i\rightsquigarrow j,
\]
if either \(i=j\) or \(P^k(i,j)>0\) for some \(k\geq1\). Similarly, we say that \(\partial\) is accessible from \(i\) if $\widehat P^k(i,\partial)>0$ for some \(k\geq1\). Our first standing assumption ensures that absorption is possible from every state.

\begin{hypothesis}\label{hyp:main}
The cemetery state \(\partial\) is accessible from every \(i\in S\).
\end{hypothesis}

For every subset \(I\subset S\), we define 
\[
\overline I:=\{j\in S:\ \exists i\in I,\ i\rightsquigarrow j\},
\]
that is, the set of states accessible from \(I\). We say that \(I\) is \emph{closed} if
\[
I=\overline I.
\]

Let \(\leftrightsquigarrow\) denote the equivalence relation on \(S\) defined by
\[
i\leftrightsquigarrow j
\quad\Longleftrightarrow\quad
i\rightsquigarrow j
\ \text{ and }\
j\rightsquigarrow i.
\]
The state space therefore decomposes into a disjoint finite union, which we
write as
\begin{equation}\label{eq:classes}
    S=\bigsqcup_{\alpha=1}^L S_\alpha,
\end{equation}
where \(L\geq1\) and \(S_1,\ldots,S_L\) are the communicating classes.

For a state \(i\in S\) and a class \(S_\alpha\), we write 
\(
i\rightsquigarrow S_\alpha
\)
if \(i\rightsquigarrow j\) for some \(j\in S_\alpha\). For two communicating
classes \(S_\beta\) and \(S_\alpha\), we write
\(
S_\beta\rightsquigarrow S_\alpha
\)
if some, and hence every, state of \(S_\beta\) can reach some state of
\(S_\alpha\).

For each class \(S_\alpha\), we denote by
\[
P_{S_\alpha}:=(P(i,j))_{i,j\in S_\alpha}
\]
the restriction of \(P\) to \(S_\alpha\). By construction, \(P_{S_\alpha}\) is an irreducible nonnegative matrix. We call its spectral radius the \emph{Perron value} of \(S_\alpha\) and denote it by \(\rho_\alpha\). By the Perron--Frobenius theorem, \(\rho_\alpha\) is a nonnegative, algebraically simple eigenvalue of \(P_{S_\alpha}\), with strictly positive left and right eigenvectors. Hypothesis~\ref{hyp:main} implies that
\[
0\leq \rho_\alpha<1.
\]
The strict inequality \(\rho_\alpha<1\) follows from the accessibility of the cemetery state from every state. Indeed, for every communicating class, the substochastic irreducible matrix \(P_{S_\alpha}\) is not conservative. Consequently, for some \(n\geq1\),
\[
\|(P_{S_\alpha})^n\|_\infty<1.
\]
Gelfand's formula then implies that the Perron eigenvalue \(\rho_\alpha\) is strictly smaller than \(1\).

\medskip

The accessibility relation induces a partial ordering on the classes $(S_\alpha)$ defined by
\[
S_\alpha\preccurlyeq S_\beta
\quad\Longleftrightarrow\quad
S_\beta\rightsquigarrow S_\alpha
\quad\Longleftrightarrow\quad
S_\alpha\subseteq\overline{S_\beta}.
\]
Equivalently, $S_\alpha\preccurlyeq S_\beta$ if and only if $S_\alpha\cap \overline{S_\beta}\neq\emptyset$. As usual, we write $S_\alpha\prec S_\beta$ when $S_\alpha\preccurlyeq S_\beta$ and $\alpha\neq \beta$.

\medskip

Following \cite{pollett2008quasi}, a class $S_\beta$ is called \emph{maximal} whenever
\[
S_\alpha\prec S_\beta \Rightarrow \rho_\alpha<\rho_\beta.
\]
Note that a class $S_\beta$ for which $S_\beta=\overline{S_\beta}$ is always maximal. 
\begin{lemma}[Positivity of maximal Perron values]
\label{lem:positive-maximal-rho}
Let \(S_\alpha\) be a maximal class. Then
\[
\rho_\alpha>0.
\]
\end{lemma}
\begin{proof}
Suppose, by contradiction, that \(\rho_\alpha=0\). If \(S_\alpha\) could reach
another communicating class \(S_\beta\), then \(S_\beta\prec S_\alpha\), and
maximality of \(S_\alpha\) would imply \(\rho_\beta<\rho_\alpha=0\), contradicting the non-negativity of Perron values. Hence \(S_\alpha\) cannot
reach any other communicating class, and therefore
\(
S_\alpha=\overline{S_\alpha}.
\) Since \(S_\alpha\) is closed, for every \(i\in S_\alpha\),
\[
\sum_{j\in S_\alpha}P(i,j)
=
\sum_{j\in S}P(i,j)
>0.
\]
Thus \(P_{S_\alpha}\) is a nonzero (every row sum is positive), irreducible, nonnegative matrix. By the
Perron--Frobenius theorem, its Perron value is strictly positive, contradicting
\(\rho_\alpha=0\).
\end{proof}

This leads to our second standing assumption.

\begin{hypothesis}\label{hyp:main-two}
The Perron values associated with the maximal classes are pairwise distinct.
\end{hypothesis}

No distinctness condition is imposed when at least one of the two classes is
non-maximal.

\begin{remark}[Genericity]
\label{rem:genericity}
Appendix~\ref{app:genericity} proves that, under the mild condition imposed
there on the fixed incidence graph, Hypothesis~\ref{hyp:main-two} holds on a
relatively open, dense subset of the corresponding incidence stratum, of full
relative Lebesgue measure.
\end{remark}

Let  $\kappa \in \{1,\ldots, L\}$ be the number of maximal classes. We relabel the classes in such a way that
\bdes
\iti the maximal classes are $S_1,\ldots,S_\kappa$;
\itii $0<\rho_1<\rho_2<\cdots<\rho_\kappa$.
\edes

\begin{lemma}[A reachable maximal class with no smaller Perron value]
\label{lem:route-max}
For every \(\alpha\in\{1,\ldots,L\}\), there exists
\(\beta\in\{1,\ldots,\kappa\}\) such that
\[
S_\alpha\rightsquigarrow S_\beta
\qquad\text{and}\qquad
\rho_\beta\geq\rho_\alpha.
\]
\end{lemma}
\begin{proof}
Fix \(\alpha\) and consider the set
\[
\mathcal A_\alpha
:=
\left\{
S_\gamma:
S_\alpha\rightsquigarrow S_\gamma
\text{ and }
\rho_\gamma\geq\rho_\alpha
\right\}.
\]
This set is nonempty, since \(S_\alpha\in\mathcal A_\alpha\). Because there
are finitely many communicating classes and the strict accessibility relation
has no cycles, we may choose \(S_\beta\in\mathcal A_\alpha\) such that no
distinct class in \(\mathcal A_\alpha\) is accessible from \(S_\beta\). We claim that \(S_\beta\) is maximal. Otherwise, by the definition of maximality, there would exist a class \(S_\delta\prec S_\beta\) such that \(\rho_\delta\geq\rho_\beta.\) Since \(S_\alpha\rightsquigarrow S_\beta\rightsquigarrow S_\delta\) and
\(\rho_\delta\geq\rho_\beta\geq\rho_\alpha\), we would have
\(S_\delta\in\mathcal A_\alpha\), contradicting the choice of \(S_\beta\).
Thus \(S_\beta\) is maximal, and the conclusion follows from the definition
of \(\mathcal A_\alpha\).
\end{proof}

Lemma~\ref{lem:route-max} implies that
\(\rho_\kappa=\max\{\rho_\alpha:\ \alpha\in\{1,\ldots,L\}\}\).
It also implies that \(S_1=\overline{S_1}\). 
% Indeed, suppose that \(S_1\)
% reaches a distinct communicating class \(S_\gamma\). By
% Lemma~\ref{lem:route-max}, the class \(S_\gamma\) reaches a maximal class
% \(S_\beta\). Necessarily \(S_\beta\neq S_1\), since otherwise \(S_1\) and
% \(S_\gamma\) would be mutually accessible. The maximality of \(S_1\)
% then gives \(\rho_\beta<\rho_1\). This contradicts the ordering of the maximal
% classes, which gives \(\rho_\beta\geq\rho_1\).
However there may exist some $\alpha >\kappa$ with $\rho_\alpha <\rho_1$. 

The conditional definition of a QSD given in the introduction is equivalent, in
the present finite discrete-time setting, to the following one-step relation: a
probability measure \(\mu\in\Delta\) is a QSD for \(P\) if and only if
\[
\mu=\frac{\mu P}{\mu P\mathbf 1}.
\]
Equivalently, setting \(\rho:=\mu P\mathbf 1\), one has \(\mu P=\rho\mu.\)
\medskip

The following result combines two ingredients. The first, due to Pollett and
van Doorn \cite[Theorems~4.1--4.2]{pollett2008quasi}, establishes,  under
Hypothesis~\ref{hyp:main-two}, a one-to-one correspondence between maximal
classes and quasi-stationary distributions. The second, due to Schneider
\cite{schneider1986influence}, provides the nonnegative right eigenvectors
associated with the Perron values of distinguished classes (defined in
Proposition~\ref{prop:setqsd}~\eqref{prop:setqsd-item-distinguished}); their
supports describe precisely the states from which the corresponding class is
accessible.

\begin{proposition}[Pollett--van Doorn and Schneider]
\label{prop:setqsd}
Assume that \(P\) satisfies Hypotheses~\ref{hyp:main} and~\ref{hyp:main-two}.
Then the following statements hold.
\ben[label=\textup{(\roman*)}, ref=\roman*, font=\textbf]

\item For every \(\alpha\in\{1,\ldots,\kappa\}\), there exists a unique QSD
\(\nu_\alpha\in\Delta\) associated with \(S_\alpha\), characterized by
\[
\nu_\alpha P=\rho_\alpha\nu_\alpha.
\]
Moreover,
\[
\supp(\nu_\alpha)=\overline{S_\alpha}.
\]

\item \label{prop:setqsd-item-qsd-set}
The set of QSDs is exactly
\[
\{\nu_1,\ldots,\nu_\kappa\}.
\]

\item \label{prop:setqsd-item-distinguished} Let \(S_\alpha\) be any communicating class of \(P\). Assume that
\(S_\alpha\) satisfies Schneider's distinguished-class condition, namely
\[
\rho_\alpha>\rho_\beta
\quad\text{for every communicating class }S_\beta\neq S_\alpha
\text{ such that }S_\beta\rightsquigarrow S_\alpha.
\]
Then there exists a right eigenvector
\(h_\alpha\in\mathbb R_+^S\setminus\{0\}\) such that
\[
Ph_\alpha=\rho_\alpha h_\alpha.
\]
Moreover, this nonnegative right \(\rho_\alpha\)-eigenvector is unique up to a
positive scalar and
\[
\supp(h_\alpha)
=
\{i\in S:\ i\rightsquigarrow S_\alpha\}.
\]

\een
\end{proposition}

A nonnegative right eigenvector of \(P\) associated with \(\rho_\kappa\)
need not be unique and need not be strictly positive on \(S\). The following
lemma records the support properties needed below. In particular, every state
in the support of such an eigenvector can reach \(S_\kappa\).

\begin{lemma}[Support properties of a top right eigenvector]
\label{lem:right-perron-support}
There exists a vector \(h\in\mathbb R^S\), with \(h\geq0\) and \(h\neq0\),
such that
\[
Ph=\rho_\kappa h.
\]
For any such vector, let \(U:=\supp(h)\). Then:
\begin{enumerate}[label=\textup{(\roman*)}]
\item \(U\) is a union of communicating classes and \(U^c\) is closed;

\item if \(\nu\in\Delta\) satisfies
\[
\nu P=\rho_\nu\nu
\qquad\text{with}\qquad
\rho_\nu\neq\rho_\kappa,
\]
then
\[
\nu h=0
\qquad\text{and}\qquad
\nu(U)=0;
\]

\item for every \(i\in U\), there exists an index
\(\beta(i)\in\{1,\ldots,L\}\) such that \(S_{\beta(i)}\subset U\),
the class \(S_{\beta(i)}\) is accessible from \(i\) along a finite
\(P\)-path contained in \(U\), and \(S_{\beta(i)}\) reaches no distinct
communicating class contained in \(U\);

\item the class \(S_{\beta(i)}\) in \textup{(iii)} satisfies
\[
\rho_{\beta(i)}=\rho_\kappa,
\qquad
\min_{x\in S_{\beta(i)}}h(x)>0;
\]

\item every \(i\in U\) can reach \(S_\kappa\).
\end{enumerate}
\end{lemma}

\begin{proof}
Lemma~\ref{lem:route-max} implies that, for every communicating class
\(S_\alpha\), there is a maximal class \(S_\beta\) such that \(\rho_\alpha\leq\rho_\beta\leq\rho_\kappa.\)
Hence
\[
\max_{1\leq\alpha\leq L}\rho_\alpha=\rho_\kappa.
\]
After a permutation of the states, \(P\) is block triangular with diagonal
blocks \(P_{S_1},\ldots,P_{S_L}\). The spectrum of \(P\) is therefore the
union of the spectra of these diagonal blocks, and consequently
\[
\operatorname{spr}(P)
=
\max_{1\leq\alpha\leq L}\rho_\alpha
=
\rho_\kappa.
\]
The Perron--Frobenius theorem for nonnegative matrices then provides a
nonzero vector \(h\geq0\) satisfying
\[
Ph=\rho_\kappa h.
\]

Suppose that \(h(x)=0\). Then
\[
0=(Ph)(x)=\sum_{y\in S}P(x,y)h(y).
\]
Every term in this sum is nonnegative. Therefore, if \(P(x,y)>0\), then
\(h(y)=0\). Repeating this argument along a path shows that every state
accessible from \(x\) also belongs to \(U^c\). Thus \(U^c\) is closed.

If a communicating class contains a state of \(U^c\), all states of that
class are accessible from this state and therefore also belong to \(U^c\).
Hence every communicating class is contained either in \(U\) or in \(U^c\).
This proves~\textup{(i)}.

Let \(\nu\in\Delta\) satisfy \(\nu P=\rho_\nu\nu\). Using
\(Ph=\rho_\kappa h\), we obtain
\[
\rho_\nu\nu h
=
(\nu P)h
=
\nu(Ph)
=
\rho_\kappa\nu h.
\]
If \(\rho_\nu\neq\rho_\kappa\), it follows that \(\nu h=0\). Since
\(h(x)>0\) for every \(x\in U\), 
% and all the terms in
% \[
% \nu h=\sum_{x\in U}\nu(x)h(x)
% \]
% are nonnegative, one must have \(\nu(x)=0\) for every \(x\in U\). 
% Thus
\(\nu(U)=0\), proving~\textup{(ii)}.

Fix \(i\in U\), and consider the finite collection of communicating classes
contained in \(U\) and accessible from \(i\). Among these classes, choose
\(S_{\beta(i)}\) so that it reaches no distinct class in the same collection.
Such a class exists because the directed graph of communicating classes is
finite and has no directed cycles.

The class \(S_{\beta(i)}\) is accessible from \(i\). Moreover, a path from
\(i\) to \(S_{\beta(i)}\) cannot leave \(U\): once a path enters the closed
set \(U^c\), it cannot subsequently return to \(U\). Hence the path may be
chosen entirely inside \(U\). 
This proves~\textup{(iii)}.

We now restrict the eigenvector equation to \(S_{\beta(i)}\). Fix
\(x\in S_{\beta(i)}\). If \(P(x,y)>0\) and \(y\notin S_{\beta(i)}\), then
either \(y\notin U\), in which case \(h(y)=0\), or \(y\in U\), in which case
the communicating class containing \(y\) is a distinct class contained in
\(U\) and accessible from \(S_{\beta(i)}\). The latter possibility is
excluded by the choice of \(S_{\beta(i)}\). Therefore
\[
P_{S_{\beta(i)}}\,h|_{S_{\beta(i)}}
=
\rho_\kappa\,h|_{S_{\beta(i)}}.
\]
Since \(S_{\beta(i)}\subset U\), the vector \(h|_{S_{\beta(i)}}\) is
strictly positive. The Perron--Frobenius theorem applied to the irreducible
matrix \(P_{S_{\beta(i)}}\) gives
\[
\rho_{\beta(i)}=\rho_\kappa.
\]
Since \(S_{\beta(i)}\) is finite and \(h>0\) on this class,
\[
\min_{x\in S_{\beta(i)}}h(x)>0.
\]
This proves~\textup{(iv)}.

Finally, Lemma~\ref{lem:route-max}, applied to \(S_{\beta(i)}\), provides a
maximal class accessible from \(S_{\beta(i)}\) whose Perron value is at least
\(\rho_{\beta(i)}=\rho_\kappa\). Since \(\rho_\kappa\) is the largest Perron
value, this maximal class also has Perron value \(\rho_\kappa\). By
Hypothesis~\ref{hyp:main-two}, the unique maximal class with this Perron value
is \(S_\kappa\). Hence
\[
S_{\beta(i)}\rightsquigarrow S_\kappa.
\]
Together with \(i\rightsquigarrow S_{\beta(i)}\), this proves that
\(i\rightsquigarrow S_\kappa\), establishing~\textup{(v)}.
\end{proof}

\medskip{}

With the state space \(S\) and the sub-Markovian kernel \(P\) fixed, the
reinforced process introduced in the previous section is determined by the
choice of a possibly random initial condition \((X_0,\mu_0)\) and a deterministic
step-size sequence \((\gamma_n)_{n\geq1}\).

We impose throughout the following standing assumption on the initial support.

\begin{hypothesis}[Initial support]
\label{hyp:initial-support}
The initial condition satisfies, almost surely,
\[
S=\overline{\supp(\mu_0)}
\qquad\text{and}\qquad
X_0\in\supp(\mu_0).
\]
\end{hypothesis}

We now turn to the assumptions on the step-size sequence \((\gamma_n)\).

\begin{hypothesis}[Step-size sequence]
\label{hyp:gamma}
The sequence \((\gamma_n)_{n\geq1}\subset(0,1)\) is eventually non-increasing
and satisfies
\[
\sum_{n\geq1}\gamma_n=\infty,
\qquad
\lim_{n\to\infty}\gamma_n\log n=0.
\]
\end{hypothesis}
These assumptions are classical in stochastic approximation theory and are in
the spirit of those used in \cite{BC2015}.

\medskip

The definition of the reinforced process associated with the kernel $P$, the initialization $(X_0,\mu_0)$ and the step-size sequence $(\gamma_n)$ is given by equation \eqref{eq:defprocessus}. 
The recursion defining \((\mu_n)\) admits a natural interpretation through an
explicit weighted occupation measure representation. Let
\((r_n)_{n\geq0}\), \((w_n)_{n\geq0}\), and \((W_n)_{n\geq0}\) be the deterministic
sequences defined by
\begin{equation}
\label{eq:defrn}
r_0:=1,
\qquad
r_n:=\prod_{k=1}^n(1-\gamma_k),
\qquad n\geq1,
\end{equation}
and
\begin{equation}
\label{eq:defwn}
w_0:=1,
\qquad
w_n:=\frac{\gamma_n}{r_n},
\quad n\geq1,
\qquad
W_n:=\sum_{k=0}^n w_k=\frac1{r_n},
\quad n\geq0.
\end{equation}
The measure \(\mu_n\) defined by \eqref{eq:defprocessus} admits the following explicit expression
\begin{equation}
\label{eq:expressionexplicitemu}
\mu_n
=
\frac1{W_n}
\left(
w_0\mu_0+\sum_{k=1}^n w_k\delta_{X_k}
\right).
\end{equation}

\begin{remark} \label{rem:casusuel}
    In the case where $\gamma_n=\frac{1}{n+1}$ and $\mu_0=\delta_{X_0}$, the weights $w_n$ are constant equal to $1$ and \(\mu_n\) coincides with the usual empirical occupation measure of the process
\[
\mu_n=\frac1{n+1}\sum_{k=0}^n\delta_{X_k}.
\]
\end{remark}

We now state the three main results of the paper. The first result describes the asymptotic behaviour of the reinforced dynamics independently of the importance assigned to the different samples through the weights \((w_k)\). It shows that the process always converges almost surely toward one of the QSDs of the model, and moreover that the QSD associated with the largest Perron value is selected with positive probability.

\begin{theorem}[General reducible setting]
\label{th:main}
Assume Hypotheses~\ref{hyp:main}, ~\ref{hyp:main-two}, ~\ref{hyp:initial-support} and ~\ref{hyp:gamma}. Let \((\mu_n)\) be defined by \eqref{eq:defprocessus}. Then \((\mu_n)_{n\geq0}\) converges almost surely to one of the QSDs \(\nu_1,\ldots,\nu_\kappa\) described in
Proposition~\ref{prop:setqsd}\,\textup{(\ref{prop:setqsd-item-qsd-set})}. Equivalently,
\[
\sum_{\alpha=1}^{\kappa}
\Pr\Bigl(\lim_{n\to\infty}\mu_n=\nu_\alpha\Bigr)=1.
\]
Moreover,
\[
\Pr\Bigl(\lim_{n\to\infty}\mu_n=\nu_\kappa\Bigr)>0.
\]
\end{theorem}

\medskip

The asymptotic behaviour of the reinforced dynamics strongly depends on the
long-time persistence of the past occupation measure in the reinforcement
mechanism and as such, is governed by the
summability of the inverse cumulative weights.

\begin{theorem}[Weak reinforcement regime]
\label{th:weak-regime}
Assume Hypotheses~\ref{hyp:main}, \ref{hyp:main-two},
\ref{hyp:initial-support}, and \ref{hyp:gamma}. Let \((\mu_n)\) be defined by
\eqref{eq:defprocessus}. In the weak reinforcement regime where
 \begin{equation}\label{eq:sum-rn=infty}
         \sum_{n\geq0}\frac1{W_n}=\infty,
    \qquad\text{equivalently}\qquad
    \sum_{n\geq0}r_n=\infty,
    \end{equation}
we have
\[
\mathbb P\Bigl(\lim_{n\to\infty}\mu_n=\nu_\kappa\Bigr)=1.
\]
\end{theorem}
\begin{remark}
The empirical case \(\gamma_n=1/(n+1)\) is covered by Theorem~\ref{th:weak-regime}. Indeed, in this case
\(
r_n
=
\frac1{n+1},
\)
and therefore  $\sum_{n\geq0}r_n = \infty.$
\end{remark}
Finally, the third theorem treats the opposite regime where the inverse cumulative weights are summable, in which case, trapping phenomena may occur with positive probability.

\begin{theorem}[Strong reinforcement regime]
\label{thm:strong-regime}
Assume Hypotheses~\ref{hyp:main}, \ref{hyp:main-two},
\ref{hyp:initial-support}, and \ref{hyp:gamma}. Let \((\mu_n)\) be defined by
\eqref{eq:defprocessus}. In the strong reinforcement regime where
  \begin{equation}\label{eq:sum-rn-finite}
    \sum_{n\geq0}\frac1{W_n}<\infty,
    \qquad\text{equivalently}\qquad
    \sum_{n\geq0}r_n<\infty,
    \end{equation}
the following holds for every \(\alpha\in\{1,\ldots,\kappa\}\),
\[
\mathbb P\Bigl(
\{X_n\in \overline{S_\alpha} \ \text{ for all sufficiently large } \ n\}
 \cap 
\{\mu_n\to \nu_\alpha\}
\Bigr)>0.
\]
Moreover, for every \(\alpha\in\{1,\ldots,\kappa\}\) and every \(N\geq0\), on
the event \(\{X_N\in \overline{S_\alpha}\}\),
\begin{align}\label{eq:closed-trapping-lb}
   \mathbb P\left(
X_n\in \overline{S_\alpha}\ \forall n\geq N
\,\middle|\,
\mathcal F_N
\right)
\geq
\prod_{n\geq N}
\left(
1-\mu_N(S\setminus \overline{S_\alpha})\frac{r_n}{r_N}
\right)>0. 
\end{align}
\end{theorem}

\begin{remark}[Polynomial weights]
Let
\[
w_0=1,
\qquad
w_k=k^q,
\quad k\geq1,
\qquad q\geq-1.
\]
The associated step sizes are \(\gamma_n=\frac{w_n}{W_n}.\)
They satisfy Hypothesis~\ref{hyp:gamma}. If \(q\in[-1,0]\), then
Theorem~\ref{th:weak-regime} applies. If \(q>0\), then
Theorem~\ref{thm:strong-regime} applies.
\end{remark}

\begin{remark}[Polynomial step-sizes away from the critical scale]
Assume Hypothesis~\ref{hyp:gamma} and, in addition,
\[
\gamma_n\sim\frac{A}{n^\alpha},
\qquad
A>0,
\qquad
0<\alpha\leq1.
\]
If \(\alpha=1\) and \(0<A<1\), then
Theorem~\ref{th:weak-regime} applies. If \(0<\alpha<1\), or if
\(\alpha=1\) and \(A>1\), then Theorem~\ref{thm:strong-regime} applies.
At the critical value \((\alpha,A)=(1,1)\), the asymptotic equivalence alone
does not determine the summability of \((r_n)\).
\end{remark}

\section{General QSD selection principle in a reducible setting}
\label{sec:general-proof}

Throughout this section, which is devoted to the proof of Theorem~\ref{th:main}, we work under
Hypotheses~\ref{hyp:main}, \ref{hyp:main-two},
\ref{hyp:initial-support}, and~\ref{hyp:gamma}.
The proof combines stochastic approximation, which controls the possible
asymptotic limits of the reinforced process, with a deterministic description
of the basins of attraction of the limiting dynamics.
\medskip

Let \(G\) be the \emph{Green kernel} on \(S\) defined by
\[
G:=\sum_{k=0}^\infty P^k=(I-P)^{-1},
\]
where \(P\) is the sub-Markovian matrix defined in \eqref{def:sub-markov-matrix}. With the notation introduced above for the killed chain, one has
\[
G(i,j)
=
\widehat{\mathbb E}_i
\left(
\sum_{k\geq0}\mathbf 1_{\{\widehat X_k=j\}}
\right),
\qquad
G\mathbf 1(i)
=
\widehat{\mathbb E}_i(\tau_\partial).
\]
It follows from Hypothesis~\ref{hyp:main} that for all \(i,j\in S, \ G(i,j)\leq G\mathbf 1(i)<\infty\).
Observe also that $i\rightsquigarrow j$ if and only if $G(i,j)>0$.

\subsection{The Green-kernel flow and stochastic approximation}

Let
\[
E^1 := \Bigl\{\mu\in\RR^S:\ \sum_{i\in S}\mu(i)=1\Bigr\}
\]
be the affine space spanning $\Delta$ and
\[
E^0 := T\Delta=\Bigl\{\mu\in\RR^S:\ \sum_{i\in S}\mu(i)=0\Bigr\}
\]
its tangent space.
Let $K : E^1 \to \RR^{S \times S}$ be the map defined by formula (\ref{eq:defnoyauressamp}) but extended to all $\mu \in E^1.$ Clearly, by the definition of $K$ and Hypothesis \ref{hyp:main}, 
\bdes
\iti $K$ is smooth,
\itii for all $\mu \in \Delta,$ $K_\mu$ is an indecomposable Markov transition matrix (meaning that it has a unique - possibly periodic - recurrence class),  and
\itiii its unique invariant probability is given by $$\pi(\mu) = \frac{\mu G}{\mu G \mathbf 1}.$$
\edes
Let $F: E^1 \to E^0$ be any smooth bounded vector field such that, for all $\mu \in \Delta,$
\begin{align}\label{def:F(mu)}
    F(\mu) := - \mu + \pi(\mu).
\end{align}

We let $\Phi = \{\Phi_t\}_{t \in \RR}$ denote the flow induced by $F.$

The following proposition identifies the affine interpolation of \((\mu_n)\) as
an asymptotic pseudo-trajectory (APT) of the flow \(\Phi\). In the empirical
case \(\gamma_n=1/(n+1)\), this follows from \cite[Proposition~3.3]{B97}.
For general step-sizes, the argument is based on the Poisson equation
decomposition used in \cite[Lemma~2.4]{BC2015} combined with Hypothesis~\ref{hyp:gamma}.

\begin{proposition}\label{prop:APT}
Let \(\mu(\cdot):\Rp\to\Delta\) be the piecewise affine interpolation of \((\mu_n)\), defined by \(\mu(\tau_n)=\mu_n\), where \(\tau_0=0\) and
\[
\tau_n := \sum_{k=1}^n\gamma_k,
\qquad n\geq 1.
\]
Then, with probability one, \(\mu(\cdot)\) is an APT of \(\Phi\). That is,
\[
\lim_{t\to\infty}\sup_{0\leq s\leq T} \left\|\Phi_s(\mu(t))-\mu(t+s) \right\|=0
\qquad \forall T>0.
\]
\end{proposition}

A key property of (relatively compact) asymptotic pseudo-trajectories is that their limit sets can be precisely described in terms of {\em internally chain-transitive sets}; a key notion (originally introduced  by Bowen and Conley~\cite{conley1978isolated}) in dynamical systems theory.
There are several equivalent definitions of {\em internally chain-transitive sets} (see e.g~\cite{BH96} and \cite{B99}) but we focus here on a characterization that will prove to be very convenient.

Let $L \subset \Delta$ be a compact set. We say that $L$ is {\em invariant}, respectively {\em positively invariant} if   $\Phi_t(L) =  L$ for all $t \in \RR,$ respectively $\Phi_t(L) \subset L$ for all $t \geq 0.$  Note that $\Delta$ is  positively invariant  but not invariant. Given  a compact positively invariant set $L,$ a subset $A \subset L$ is called {\em an attractor for $\Phi|_L$} if $A$ is compact invariant and has a neighbourhood $U$ such that
$$\lim_{t \rar \infty} \sup_{x \in U \cap L } \mathsf{dist}(\Phi_t(x),A) = 0.$$ 
An {\em internally chain-transitive} set is a compact connected invariant set $L$ that  has no proper attractor; in other words, the only attractor for $\Phi|_L$ is $L.$
The next result rephrases the {\em limit set theorem}, originally proved in \cite{Ben96} for stochastic approximation (SA) processes and later in \cite{BH96} for asymptotic pseudo-trajectories. We refer the reader to this reference and \cite{B99} for more details on the dynamics of asymptotic pseudo-trajectories and their relation to chain-recurrence and internally chain-transitive sets.

\begin{proposition}
\label{prop:limitset}
The limit set of $\{\mu(t):t\geq 0\}$ (or equivalently the limit set of $(\mu_n)$) is almost surely an internally chain-transitive set for $\Phi$.
\end{proposition}

We now turn to the deterministic part of the argument. The stochastic approximation theorem tells us that every almost sure limit set of $(\mu_n)$ must be an internally chain-transitive set for the flow $\Phi$. It remains to understand the geometry of $\Phi$ on the simplex. The key point is that the QSDs are exactly the equilibria of the flow, and that each of them attracts exactly one natural stratum of the simplex.

\begin{lemma}
Let
\[
\mathsf{Eq}(F)
:=
\{\mu\in\Delta:\ F(\mu)=0\}
\]
denote the set of equilibria of \(F\). Then
\[
\mathsf{Eq}(F)
=
\{\nu_1,\ldots,\nu_\kappa\}.
\]
\end{lemma}

\begin{proof}
The relation $(I-P)G=G(I-P)=I$ shows that
\[
\mu P=\lambda\mu
\qquad\Longleftrightarrow\qquad
\mu G=\frac{1}{1-\lambda}\mu.
\]
Thus the eigenvectors of $P$ in $\Delta$ are exactly the eigenvectors of $G$ in $\Delta$. Since the QSDs are precisely the eigenvectors of $P$ in $\Delta$, and the equilibria of $F$ are precisely the fixed points of $\mu\mapsto \mu G/(\mu G{\bf 1})$, the two sets coincide.
\end{proof}

Given a set $I\subset S$, let
\[
\Delta^I:=\{\mu\in\Delta:\ \supp(\mu)\subset I\}.
\]
Geometrically, $\Delta^I$ is the face of the simplex supported on $I$.
\begin{lemma}
If $I\subset S$ is closed, then $\Delta^I$ is positively invariant.
\end{lemma}
\begin{proof}
Let $\mu\in\Delta^I$. Then, for all $i\notin I, \  \mu(i)=0$. Set $i\notin I$. Since $I$ is closed, for every $j\in I$ one has $j\not\rightsquigarrow i$, hence
$G(j,i)=0$. Therefore
\[
\mu G(i)=\sum_{j\in I}\mu(j)G(j,i)=0.
\]
It follows that for all $i\notin I, \ F_i(\mu)=0$, so the vector field is tangent to $\Delta^I$.
\end{proof}

For each $\beta=1,\ldots,\kappa$, let $I_\beta\subset S$ denote the largest
closed set such that
\bdes
\iti $S_\beta\subset I_\beta$;
\itii if a communicating class $S_\gamma$ is contained in $I_\beta$, then
$\rho_\gamma\leq\rho_\beta$.
\edes

In words, $I_\beta$ is the largest closed region of the state space whose
accessible maximal class of highest Perron value is $S_\beta$. Clearly,
\[
I_1\subset I_2\subset\cdots\subset I_\kappa=S,
\]
and therefore
\[
\Delta^{I_1}\subset\Delta^{I_2}\subset\cdots
\subset\Delta^{I_\kappa}=\Delta.
\]
For convenience, we set \(I_0:=\emptyset\), so that
\(\Delta^{I_0}=\emptyset\).

\begin{figure}[ht]
    \centering
    \includegraphics[width=0.7\textwidth]{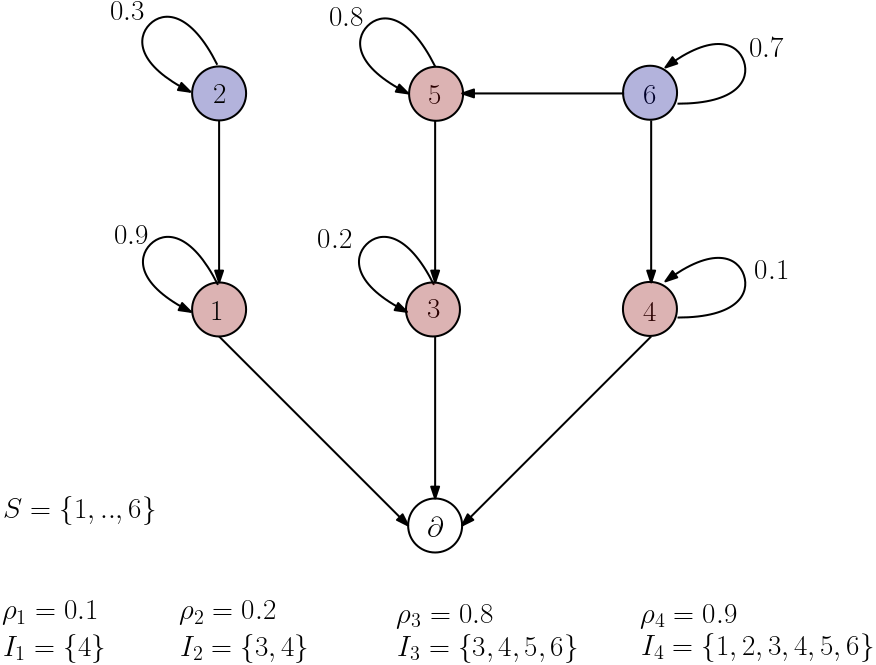}
\caption[Example of the recursive construction of the sets \(I_\alpha\)]{
     \centering \textbf{Example of the recursive construction of the sets \(I_\alpha\).}\\[3pt]
    \begin{minipage}{0.92\linewidth}
\justifying
  The figure shows the directed graph associated with a sub-Markov kernel
    satisfying the assumptions on the model.
    In this example every communicating class is a singleton, so the
    Perron value of a class is the weight of its self-loop. The maximal classes
    are the red vertices, ordered by increasing Perron value as
    \(S_1=\{4\}\), \(S_2=\{3\}\), \(S_3=\{5\}\), and \(S_4=\{1\}\).
    The blue vertices are not maximal. The associated recursively defined sets
    \(I_\alpha\) form an increasing sequence of closed regions. The basin of the QSD associated with \(S_\alpha\) is the stratum
\(\Delta^{I_\alpha}\setminus\Delta^{I_{\alpha-1}}\).
    \end{minipage}}
    \label{fig:grapheexemple}
\end{figure}

\begin{lemma}
\label{lem:Ialpha-char}
For every $\alpha\in\{1,\ldots,\kappa\}$, one has
\[
I_\alpha=
\{i\in S:\ i\not\rightsquigarrow S_\beta
\text{ for every }\beta>\alpha\}.
\]
In particular,
\[
I_\alpha\setminus I_{\alpha-1}
=
\{i\in I_\alpha:\ i\rightsquigarrow S_\alpha\}.
\]
\end{lemma}

\begin{proof}
Set
\[
J_\alpha
:=
\{i\in S:\ i\not\rightsquigarrow S_\beta
\text{ for every }\beta>\alpha\}.
\]
The set \(J_\alpha\) is closed by transitivity of accessibility. Moreover,
\(S_\alpha\subset J_\alpha\): otherwise \(S_\alpha\) would reach some
\(S_\beta\), \(\beta>\alpha\), contradicting the maximality of
\(S_\alpha\).
If a communicating class \(S_\gamma\subset J_\alpha\) satisfied
\(\rho_\gamma>\rho_\alpha\), Lemma~\ref{lem:route-max} would give a maximal
class \(S_\beta\) accessible from \(S_\gamma\) such that
\[
\rho_\beta\geq\rho_\gamma>\rho_\alpha.
\]
The ordering of the maximal classes would imply \(\beta>\alpha\). Since
\(J_\alpha\) is closed and \(S_\gamma\subset J_\alpha\), this would give
\(S_\beta\subset J_\alpha\), contradicting the definition of \(J_\alpha\).
Thus \(J_\alpha\) satisfies the defining properties of \(I_\alpha\), and
therefore \(J_\alpha\subset I_\alpha\).

Conversely, if \(i\in I_\alpha\) reached some \(S_\beta\) with
\(\beta>\alpha\), closedness of \(I_\alpha\) would imply
\(S_\beta\subset I_\alpha\), contradicting its defining Perron-value
property. Hence \(I_\alpha\subset J_\alpha\), proving the first identity.

Finally, every communicating class reaches a maximal class. The
characterizations of \(I_\alpha\) and \(I_{\alpha-1}\) therefore give
\[
I_\alpha\setminus I_{\alpha-1}
=
\{i\in I_\alpha:\ i\rightsquigarrow S_\alpha\}.
\]
\end{proof}

We now determine the basin of attraction of each equilibrium $\nu_\alpha$.

\begin{proposition}
\label{prop:basins}
For every \(\alpha=1,\ldots,\kappa\),
\bdes
\iti \(\{\nu_\alpha\}\) is an attractor for the restricted flow on \(\Delta^{I_\alpha}\).
\itii Its basin in \(\Delta\) is
\[
W^s(\nu_\alpha):=\{\mu\in\Delta:\ \Phi_t(\mu)\to\nu_\alpha\}
=\Delta^{I_\alpha}\setminus\Delta^{I_{\alpha-1}}.
\]
\edes
\end{proposition}

\begin{figure}[ht]
    \centering

    \begin{subfigure}{0.38\textwidth}
        \centering
        \includegraphics[width=\textwidth]{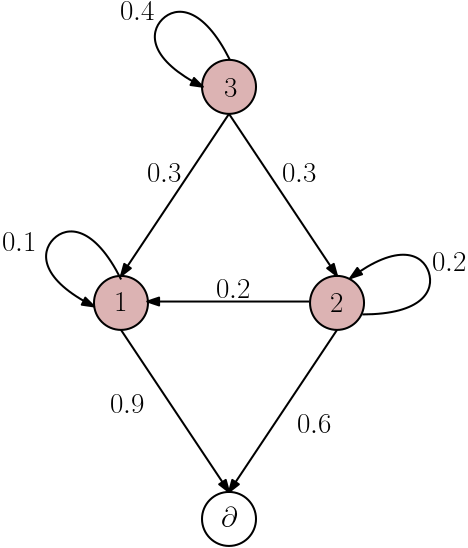}
    \end{subfigure}
    \hfill
    \begin{subfigure}{0.60\textwidth}
        \centering
        \includegraphics[width=\textwidth]{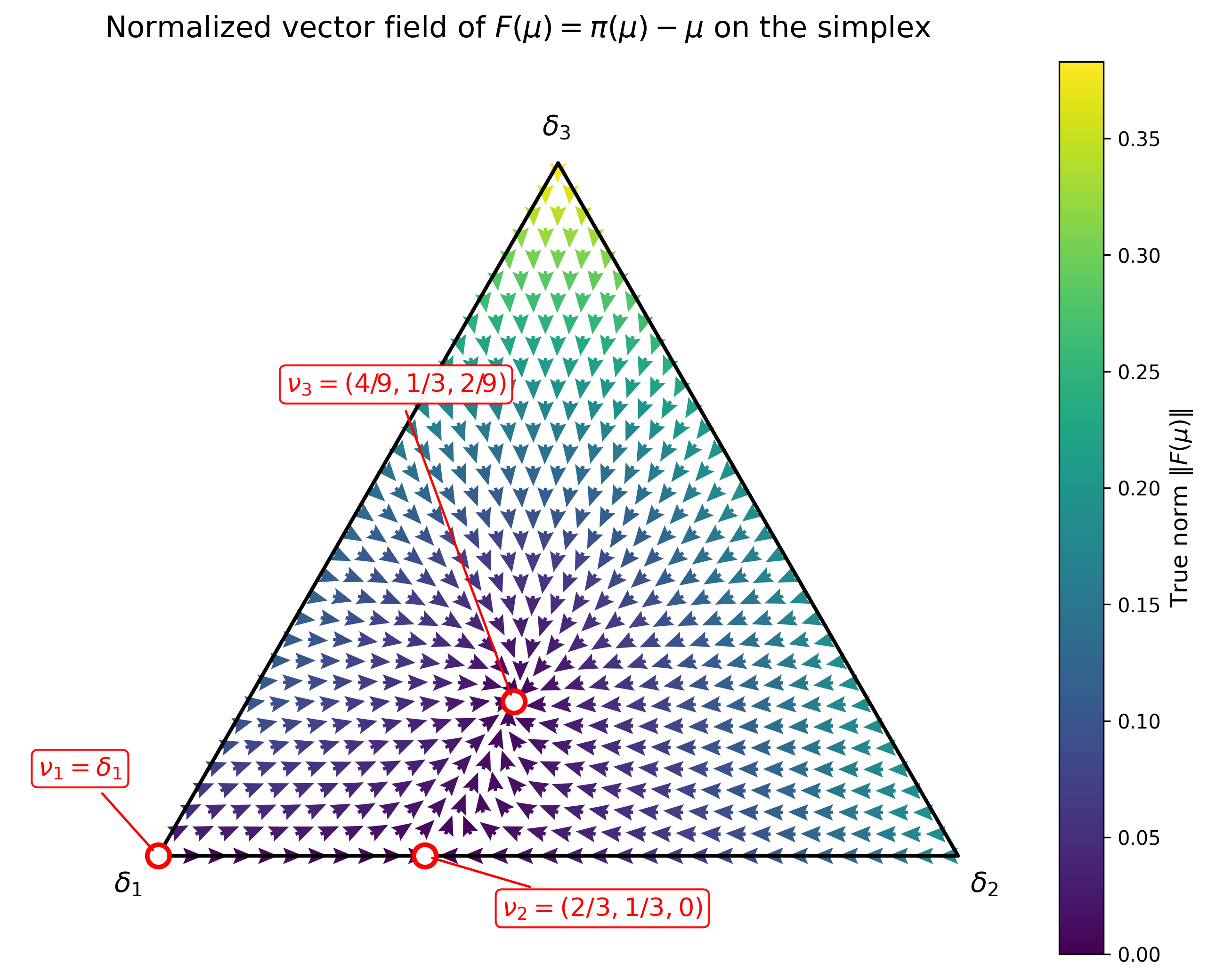}
    \end{subfigure}
    \caption[An illustration of the basins of attraction of the QSDs for a simple sub-Markovian kernel.]{
     \centering\textbf{An illustration of the basins of attraction of the QSDs for a simple sub-Markovian kernel.}\\[3pt]
    \begin{minipage}{0.92\linewidth}
\justifying
    The left panel displays a weighted graph associated with a sub-Markov kernel, while the right panel shows the corresponding vector field $F(\mu)=\pi(\mu)-\mu$ on the simplex. This example illustrates Proposition~\ref{prop:basins}.
     \end{minipage}}
    \label{fig:twopictures}
\end{figure}

\begin{proof}
Let
\[
P^{(\alpha)}:=P|_{I_\alpha\times I_\alpha},
\qquad
G^{(\alpha)}:=G|_{I_\alpha\times I_\alpha}.
\]
Since $I_\alpha$ is closed,
\[
G^{(\alpha)}
=
\sum_{n\geq0}(P^{(\alpha)})^n
=
(I-P^{(\alpha)})^{-1}.
\]

\paragraph{Step 1: Time-change and explicit formula on $\boldsymbol{\Delta^{I_\alpha}}$.}
Let
\[
\widetilde F(\mu):=\mu G-(\mu G{\bf 1})\mu,
\]
so that $F(\mu)=\widetilde F(\mu)/(\mu G{\bf 1})$, and denote by
$\widetilde\Phi$ the flow of $\widetilde F$. If $x(\cdot)$ solves
$\dot x=\widetilde F(x)$ with $x(0)=\mu$, define
\[
s_\mu(t):=\int_0^t x(u)G{\bf 1}\,du.
\]
Because $1\leq x(u)G{\bf 1}\leq M$, where
\(
M:=\max\limits_{i\in S}G{\bf 1}(i)<\infty,
\)
the map $s_\mu$ is a $C^1$-bijection of $\mathbb R_+$ and
$x(t)=\Phi_{s_\mu(t)}(\mu)$. Thus $\Phi$ and $\widetilde\Phi$ have the same
trajectories, with different time parametrizations. Since $I_\alpha$ is
closed, the flow starting from $\Delta^{I_\alpha}$ remains in this face. For
$\mu\in\Delta^{I_\alpha}$, direct differentiation gives
\[
\widetilde\Phi_t(\mu)
=
\frac{\mu e^{tG^{(\alpha)}}}
     {\mu e^{tG^{(\alpha)}}{\bf 1}},
\qquad t\geq0.
\]

\paragraph{Step 2: Expansion at the largest growth rate.}
Set
\[
\rho:=\rho_\alpha,
\qquad
g:=\frac{1}{1-\rho}.
\]
By Lemma~\ref{lem:Ialpha-char}, the spectral radius of $P^{(\alpha)}$ is
$\rho$. If $\lambda=a+ib$ is an eigenvalue of $P^{(\alpha)}$, then
$|\lambda|\leq\rho<1$ and
\[
\frac{1}{1-\rho}
-
\operatorname{Re}\left(\frac{1}{1-\lambda}\right)
=
\frac{(1-a)(\rho-a)+b^2}
     {(1-\rho)((1-a)^2+b^2)}
\geq0.
\]
Since $1-a>0$ and $a\leq|\lambda|\leq\rho$, equality holds only when
$\lambda=\rho$. The eigenvalues of $G^{(\alpha)}$ are
$(1-\lambda)^{-1}$; hence $g$ is its unique eigenvalue with maximal real
part.

By the real Jordan form, after reordering the blocks, there is an invertible real
matrix $T$ such that
\[
T^{-1}G^{(\alpha)}T
=
\begin{pmatrix}
 gI+N&0\\
 0&H
\end{pmatrix},
\]
where $N^m=0$ for some $m\geq1$, and every eigenvalue of $H$ has real part
strictly smaller than $g$. Since the exponential of a block diagonal matrix
is obtained by exponentiating each block and
\[
e^{t(gI+N)}
=
e^{gt}\sum_{r=0}^{m-1}\frac{t^r}{r!}N^r,
\]
we obtain, for $r=0,\ldots,m-1$,
\[
B_r
:=
T
\begin{pmatrix}
\dfrac{1}{r!}N^r&0\\
0&0
\end{pmatrix}
T^{-1},
\qquad
R_t
:=
T
\begin{pmatrix}
0&0\\
0&e^{tH}
\end{pmatrix}
T^{-1},
\]
and therefore
\begin{equation}
\label{eq:dominant-expansion-Ialpha}
e^{tG^{(\alpha)}}
=
e^{gt}\sum_{r=0}^{m-1}t^rB_r+R_t.
\end{equation}

If the block $H$ is absent, then $R_t=0$. Otherwise, let
\[
\omega:=\max\{\operatorname{Re}z:z\text{ is an eigenvalue of }H\}<g.
\]
Let $q$ be one less than the largest size of a Jordan block of $H$. For the
matrix norm induced by the $\ell^1$ norm on row vectors, the Jordan form of
$H$ gives a constant $C_0>0$ such that
\[
\|e^{tH}\|\leq C_0(1+t^q)e^{\omega t},
\qquad t\geq0.
\]
Set $\eta:=(g-\omega)/2$. Since
$(1+t^q)e^{-\eta t}$ is bounded on $\mathbb R_+$, submultiplicativity of the
matrix norm gives, after increasing the constant if necessary,
\begin{equation}
\label{eq:remainder-Ialpha}
\|R_t\|
\leq
Ce^{(g-\eta)t},
\qquad t\geq0.
\end{equation}

We shall also use the following consequence of the Pollett--van Doorn
classification. Since $I_\alpha$ is closed, the restricted kernel
$P^{(\alpha)}$ satisfies Hypothesis~\ref{hyp:main}, and its maximal classes
are precisely the maximal classes of $P$ contained in $I_\alpha$. Hence
Hypothesis~\ref{hyp:main-two} is inherited by the restriction, and
$S_\alpha$ is the unique maximal class of $P^{(\alpha)}$ with Perron value
$\rho_\alpha$. It follows from \cite[Theorem~4.1]{pollett2008quasi} that
every nonzero nonnegative row vector $v$ satisfying
\[
vP^{(\alpha)}=\rho_\alpha v
\]
is of the form
\[
v=c\nu_\alpha
\qquad\text{for some }c>0.
\]
Here $\nu_\alpha$ is regarded as a measure on $I_\alpha$, since
$\operatorname{supp}(\nu_\alpha)=\overline{S_\alpha}\subset I_\alpha$.

\paragraph{Step 3: Pointwise convergence on the stratum.}
Suppose first that, for $\mu\in\Delta^{I_\alpha}$, at least one of the
vectors $\mu B_0,\ldots,\mu B_{m-1}$ is nonzero. Let $k$ be the largest
index such that $\mu B_k\neq0$. Multiplying
\eqref{eq:dominant-expansion-Ialpha} by $e^{-gt}t^{-k}$ gives
\[
e^{-gt}t^{-k}\mu e^{tG^{(\alpha)}}
=
\mu B_k
+
\sum_{r=0}^{k-1}t^{r-k}\mu B_r
+
e^{-gt}t^{-k}\mu R_t.
\]
The last two terms tend to zero by \eqref{eq:remainder-Ialpha}; hence
\begin{equation}
\label{eq:leading-limit-Ialpha}
e^{-gt}t^{-k}\mu e^{tG^{(\alpha)}}
\longrightarrow
v:=\mu B_k.
\end{equation}
The vector on the left is nonnegative for $t>0$, so $v\geq0$, and $v\neq0$.
Set $B_m:=0$. The definitions of the matrices $B_r$ give
\[
B_rG^{(\alpha)}=gB_r+(r+1)B_{r+1},
\qquad 0\leq r\leq m-1.
\]
Since $\mu B_{k+1}=0$, it follows that
\[
vG^{(\alpha)}=gv.
\]
Multiplying on the right by $I-P^{(\alpha)}$ gives
\[
vP^{(\alpha)}=\rho_\alpha v.
\]
By the preceding consequence of Theorem~4.1, $v=c\nu_\alpha$ for some
$c>0$.

To pass to the normalized flow, set
\[
z_t:=e^{-gt}t^{-k}\mu e^{tG^{(\alpha)}}.
\]
Then $z_t\to c\nu_\alpha$ and, since $\nu_\alpha{\bf 1}=1$,
\[
z_t{\bf 1}\longrightarrow c.
\]
% Moreover, multiplication of both the numerator and denominator by the same
% positive scalar does not change the normalized vector, so
% \[
% \widetilde\Phi_t(\mu)
% =
% \frac{z_t}{z_t{\bf 1}}
% \longrightarrow
% \nu_\alpha.
% \]
Thus
\begin{equation}
\label{eq:pointwise-visible-Ialpha}
\widetilde\Phi_t(\mu)\longrightarrow\nu_\alpha
\end{equation}
whenever at least one $\mu B_r$ is nonzero.

We now prove that this occurs exactly when
\(
\mu\in\Delta^{I_\alpha}\setminus\Delta^{I_{\alpha-1}}.
\)
Assume first that $\alpha\geq2$ and
$\mu\in\Delta^{I_{\alpha-1}}$. Since $I_{\alpha-1}$ is closed,
\[
\mu e^{tG^{(\alpha)}}
=
\mu e^{tG^{(\alpha-1)}},
\qquad
G^{(\alpha-1)}:=G|_{I_{\alpha-1}\times I_{\alpha-1}},
\]
where both sides are viewed as row vectors on $I_\alpha$. The spectral radius
of $P|_{I_{\alpha-1}\times I_{\alpha-1}}$ is $\rho_{\alpha-1}$, so every
eigenvalue of $G^{(\alpha-1)}$ has real part strictly smaller than $g$.
Applying the same Jordan estimate to $G^{(\alpha-1)}$, there exist constants
$C_1,\eta_1>0$ such that
\[
\|\mu e^{tG^{(\alpha)}}\|_1
\leq
C_1e^{(g-\eta_1)t}.
\]
If some $\mu B_r$ were nonzero and $k$ were the largest such index, then
\eqref{eq:leading-limit-Ialpha} would converge to the nonzero vector
$\mu B_k$, whereas
\[
\left\|e^{-gt}t^{-k}\mu e^{tG^{(\alpha)}}\right\|_1
\leq
C_1t^{-k}e^{-\eta_1t}
\longrightarrow0,
\]
a contradiction. Hence
\[
\mu B_r=0,
\qquad r=0,\ldots,m-1.
\]
For $\alpha=1$, this implication is vacuous because
$\Delta^{I_0}=\emptyset$.

Conversely, let
$\mu\in\Delta^{I_\alpha}\setminus\Delta^{I_{\alpha-1}}$. By
Lemma~\ref{lem:Ialpha-char}, there exists $i\in I_\alpha$ such that
$\mu(i)>0$ and $i\rightsquigarrow S_\alpha$. Let
\[
G_{S_\alpha}:=(I-P_{S_\alpha})^{-1},
\]
and let $r_\alpha>0$ be a right Perron eigenvector of $P_{S_\alpha}$,
extended by zero to $I_\alpha$. Then
\[
G_{S_\alpha}r_\alpha=gr_\alpha.
\]
Since $i\rightsquigarrow S_\alpha$ and $r_\alpha$ is strictly positive on
$S_\alpha$,
\[
c_i
:=
\sum_{j\in S_\alpha}G^{(\alpha)}(i,j)r_\alpha(j)
>0.
\]
A path starting and ending in $S_\alpha$ cannot leave this communicating
class in between; otherwise every state visited outside the class would be
mutually accessible with $S_\alpha$. Hence
\[
G^{(\alpha)}|_{S_\alpha\times S_\alpha}=G_{S_\alpha}.
\]
It follows, for the extension of $r_\alpha$ by zero, that
\[
G^{(\alpha)}r_\alpha\geq gr_\alpha
\]
componentwise. Since $G^{(\alpha)}$ is nonnegative, iteration gives
\[
(G^{(\alpha)})^nr_\alpha
\geq
g^{n-1}G^{(\alpha)}r_\alpha,
\qquad n\geq1.
\]
Consequently,
\[
\begin{aligned}
\bigl(\delta_i e^{tG^{(\alpha)}}\bigr)r_\alpha
\geq
\sum_{n=1}^{\infty}\frac{t^n}{n!}c_ig^{n-1}=
\frac{c_i}{g}(e^{gt}-1).
\end{aligned}
\]
Since all entries are nonnegative,
\[
(\mu e^{tG^{(\alpha)}})r_\alpha
\geq
\mu(i)(\delta_i e^{tG^{(\alpha)}})r_\alpha,
\]
while, for every nonnegative row vector $z$,
\[
zr_\alpha\leq\|r_\alpha\|_\infty\|z\|_1.
\]
Therefore
\begin{equation}
\label{eq:lower-growth-Ialpha}
\|\mu e^{tG^{(\alpha)}}\|_1
\geq
\frac{\mu(i)c_i}{g\|r_\alpha\|_\infty}(e^{gt}-1).
\end{equation}
If $\mu B_r=0$ for every $r$, then
\eqref{eq:dominant-expansion-Ialpha} and
\eqref{eq:remainder-Ialpha} would instead give
\[
\|\mu e^{tG^{(\alpha)}}\|_1
=
\|\mu R_t\|_1
\leq
Ce^{(g-\eta)t},
\]
which contradicts \eqref{eq:lower-growth-Ialpha} for large $t$. Hence at least
one of the vectors $\mu B_r$ is nonzero. Combining this fact with
\eqref{eq:pointwise-visible-Ialpha}, we obtain
\begin{equation}
\label{eq:pointwise-stratum-Ialpha}
\widetilde\Phi_t(\mu)\longrightarrow\nu_\alpha
\qquad
\text{for every }
\mu\in\Delta^{I_\alpha}\setminus\Delta^{I_{\alpha-1}}.
\end{equation}

\paragraph{Step 4: Uniform attraction.}
Choose $i_0\in S_\alpha$ and set
\[
\varepsilon_0:=\frac{\nu_\alpha(i_0)}{2},
\qquad
\mathcal K_\alpha
:=
\{\mu\in\Delta^{I_\alpha}:\mu(i_0)\geq\varepsilon_0\}.
\]
Since $i_0\in I_\alpha\setminus I_{\alpha-1}$, the set
$\mathcal K_\alpha$ is contained in the stratum. Moreover, because
$\nu_\alpha(i_0)=2\varepsilon_0$, it contains a relative neighbourhood of
$\nu_\alpha$ in $\Delta^{I_\alpha}$. Define
\[
y_0:=\delta_{i_0},
\qquad
y_j:=\varepsilon_0\delta_{i_0}+(1-\varepsilon_0)\delta_j,
\quad j\in I_\alpha\setminus\{i_0\}.
\]
Every one of these finitely many measures belongs to the stratum. Moreover,
every $\mu\in\mathcal K_\alpha$ has the convex decomposition
\[
\mu
=
q_0y_0+
\sum_{j\in I_\alpha\setminus\{i_0\}}q_jy_j,
\]
where
\[
q_0:=\frac{\mu(i_0)-\varepsilon_0}{1-\varepsilon_0},
\qquad
q_j:=\frac{\mu(j)}{1-\varepsilon_0}.
\]
The coefficients are nonnegative and sum to one. Let $\mathcal Y_\alpha$ denote the finite family \((y_j)_{j\in I_\alpha}\). For a representation
$\mu=\sum_{y\in\mathcal Y_\alpha}q_y y$, set
\[
d_y(t):=ye^{tG^{(\alpha)}}{\bf 1}>0.
\]
Linearity gives the exact identity
\[
\widetilde\Phi_t(\mu)
=
\sum_{y\in\mathcal Y_\alpha}
\frac{q_yd_y(t)}
     {\sum_{z\in\mathcal Y_\alpha}q_zd_z(t)}
\widetilde\Phi_t(y).
\]
The coefficients on the right are nonnegative and sum to one. Hence
\[
\|\widetilde\Phi_t(\mu)-\nu_\alpha\|_1
\leq
\max_{y\in\mathcal Y_\alpha}
\|\widetilde\Phi_t(y)-\nu_\alpha\|_1.
\]
Every $y\in\mathcal Y_\alpha$ belongs to the stratum, and the family is
finite. Thus \eqref{eq:pointwise-stratum-Ialpha} implies
\[
\sup_{\mu\in\mathcal K_\alpha}
\|\widetilde\Phi_t(\mu)-\nu_\alpha\|_1
\longrightarrow0.
\]
Therefore $\{\nu_\alpha\}$ is an attractor for the time-changed flow on
$\Delta^{I_\alpha}$. Finally, if $t_\mu(s)$ denotes the inverse of $s_\mu(t)$, Step~1 gives
\[
\frac{s}{M}\leq t_\mu(s)\leq s,
\qquad
\Phi_s(\mu)=\widetilde\Phi_{t_\mu(s)}(\mu).
\]
These bounds are uniform in $\mu$. Hence the pointwise convergence on the
stratum and the uniform convergence on $\mathcal K_\alpha$ transfer to
$\Phi$.

\paragraph{Step 5: Identification of the basin of attraction.}
For $\beta=1,\ldots,\kappa$, the sets
$\Delta^{I_\beta}\setminus\Delta^{I_{\beta-1}}$ form a partition of
$\Delta$. Applying the preceding convergence result to each $\beta$ gives
\[
W^s(\nu_\alpha)
=
\Delta^{I_\alpha}\setminus\Delta^{I_{\alpha-1}},
\]
which completes the proof.
\end{proof}

\begin{corollary}
\label{CRset}
Every internally chain-transitive set \(L\subset\Delta\) is a singleton equilibrium. Hence every internally chain-transitive set is one of the QSDs \(\nu_1,\ldots,\nu_\kappa\).
\end{corollary}

\begin{proof}
There exists a unique \(\alpha\) such that \(L\subset\Delta^{I_\alpha}\) and \(L\not\subset\Delta^{I_{\alpha-1}}\). Take \(x\in L\cap(\Delta^{I_\alpha}\setminus\Delta^{I_{\alpha-1}})\). By Proposition~\ref{prop:basins}, \(\Phi_t(x)\to\nu_\alpha\). Since \(L\) is compact and invariant, \(\nu_\alpha\in L\). But \(\{\nu_\alpha\}\) is an attractor for \(\Phi|_{\Delta^{I_\alpha}}\), and therefore for \(\Phi|_L\). Since \(L\) has no proper attractor, \(L=\{\nu_\alpha\}\).
\end{proof}

\subsection{Proof of Theorem~\ref{th:main}}
\begin{proof}
By Proposition~\ref{prop:APT}, the affine interpolation \(\mu(\cdot)\) is
almost surely an APT of \(\Phi\). Since \(\Delta\) is compact, Proposition~\ref{prop:limitset} implies that its limit set is almost surely internally chain-transitive. By Corollary~\ref{CRset}, this limit set is a
singleton equilibrium. Hence
\[
\sum_{\alpha=1}^{\kappa}
\Pr\Bigl(\lim_{n\to\infty}\mu_n=\nu_\alpha\Bigr)=1.
\]
It remains to prove that \(\nu_\kappa\) is selected with positive probability. The Poisson-equation estimates underlying Proposition~\ref{prop:APT},
together with Hypothesis~\ref{hyp:gamma}, yield
\cite[condition~(24)]{B99}. We may therefore use the second assertion of
\cite[Theorem~7.3]{B99}. It is enough to find an open set
\(\mathcal O_\kappa\subset\Delta\) such that
\[
\overline{\mathcal O_\kappa}\subset W^s(\nu_\kappa)
\]
and, for every \(m\geq0\),
\[
\mathbb P\bigl(\exists n\geq m:\ \mu_n\in\mathcal O_\kappa\bigr)>0.
\]
Choose a nonzero vector \(h\in\mathbb R_+^S\) satisfying
\[
Ph=\rho_\kappa h,
\]
and set \(U:=\supp(h)\). Such a vector exists by
Lemma~\ref{lem:right-perron-support}, which also shows that \(U^c\) is closed.
Hypothesis~\ref{hyp:initial-support} therefore implies
\[
\mu_0(U)>0
\qquad\text{almost surely}.
\]
Indeed, otherwise \(\supp(\mu_0)\subset U^c\), and closedness would give
\(\overline{\supp(\mu_0)}\subset U^c\), contradicting
\(\overline{\supp(\mu_0)}=S\). Hence, by the weighted representation
\eqref{eq:expressionexplicitemu},
\[
\mu_n(U)\geq r_n\mu_0(U)>0,
\qquad n\geq0.
\]
Now define the open set
\[
\mathcal O_\kappa
:=
\{\mu\in\Delta:\ \mu(S_\kappa)>1/2\}.
\]
Since \(S_\kappa\subset S\setminus I_{\kappa-1}\), one has
\[
\overline{\mathcal O_\kappa}
=
\{\mu\in\Delta:\ \mu(S_\kappa)\geq1/2\}
\subset
\Delta\setminus\Delta^{I_{\kappa-1}}.
\]
By Proposition~\ref{prop:basins}, because \(I_\kappa=S\),
\[
\Delta\setminus\Delta^{I_{\kappa-1}}
=
W^s(\nu_\kappa).
\]
Thus
\[
\overline{\mathcal O_\kappa}\subset W^s(\nu_\kappa).
\]
It remains to verify the late-hitting condition. Fix a
deterministic \(m\geq0\), and work conditionally on \(\mathcal F_m\). By
Hypothesis~\ref{hyp:main}, the cemetery state is accessible from \(X_m\). Hence there is a finite \(P\)-path from \(X_m\) to a state \(z\) with
\(q(z)>0\). With positive conditional probability, the process follows this
path and then uses the resurrection part \(q(z)\mu_t\) of \(K_{\mu_t}(z,\cdot)\). Since \(\mu_t(U)>0\), it is then resurrected in
\(U\) with positive conditional probability. By
Lemma~\ref{lem:right-perron-support}\textup{(v)}, every state of \(U\) can
reach \(S_\kappa\) along a finite \(P\)-path. Since
\(K_\mu(x,y)\geq P(x,y)\), the process subsequently reaches \(S_\kappa\)
with positive conditional probability.
Once the process is in \(S_\kappa\), it can remain in \(S_\kappa\) for any
prescribed finite number of steps with positive conditional probability. Indeed,
\(P_{S_\kappa}\) is irreducible and \(\rho_\kappa>0\), hence
\[
P(i,S_\kappa)>0,
\qquad i\in S_\kappa.
\]
Let \(s\) be the time at which the above finite construction first reaches
\(S_\kappa\). If the process remains in \(S_\kappa\) from time \(s\) to time
\(N\), and if
\[
a_k:=\mu_k(S_\kappa),
\]
then, for \(s\leq k<N\),
\[
a_{k+1}=(1-\gamma_{k+1})a_k+\gamma_{k+1}.
\]
Consequently,
\[
1-a_N
=
(1-a_s)\prod_{k=s+1}^{N}(1-\gamma_k)
\leq
\prod_{k=s+1}^{N}(1-\gamma_k).
\]
Since \(\sum_k\gamma_k=\infty\), the finite stay inside \(S_\kappa\) can be
chosen long enough so that \(a_N>1/2\). Hence \(\mu_N\in\mathcal O_\kappa.\)
This proves the stronger conditional statement
\[
\mathbb P\Bigl(
\exists n\geq m:\ \mu_n\in\mathcal O_\kappa
\ \vert \ 
\mathcal F_m
\Bigr)>0
\qquad\text{almost surely}.
\]
Taking expectations, and using \(\tau_m\to\infty\), gives the required
late-hitting property for the affine interpolation. The second assertion of \cite[Theorem~7.3]{B99} therefore gives
\[
\Pr\Bigl(\lim_{n\to\infty}\mu_n=\nu_\kappa\Bigr)>0.
\]
This completes the proof.
\end{proof}

\section{Selection principle in the weak reinforcement regime}\label{sec:selection-weak-regime}

Throughout this section, we work under Hypotheses~\ref{hyp:main},
\ref{hyp:main-two}, \ref{hyp:initial-support}, and \ref{hyp:gamma}, and assume
that the process is in the weak reinforcement regime, namely
\[
\sum_{n\geq0}\frac1{W_n}=\infty,
\qquad\text{equivalently} \qquad
\sum_{n\geq0}r_n=\infty.
\]
The aim of this section is to prove Theorem~\ref{th:weak-regime}. The proof
proceeds by excluding convergence to every QSD whose Perron value is strictly
smaller than \(\rho_\kappa\). The exclusion mechanism combines deterministic
repulsion in the direction associated with the largest Perron value with
martingale estimates for the projection onto this same unstable direction.

\subsection{Overview of the argument}

By Theorem~\ref{th:main}, it is enough to exclude convergence to each lower QSD \(\nu_\alpha\), \(\alpha<\kappa\).

Fix such a QSD \(\nu\) and let
\(h\in\mathbb R_+^S\setminus\{0\}\) satisfy
\[
Ph=\rho_\kappa h.
\]
The exclusion argument combines two complementary mechanisms. The first is the
deterministic instability of the coordinate \(\mu h\) near \(\nu\). The
second is a probabilistic control of this coordinate, based on arbitrarily
late excursions into suitable communicating classes contained in
\(\supp(h)\), on which \(h\) is bounded away from zero, together with
martingale concentration over deterministic time blocks.

The first mechanism is a deterministic instability. The flow \(\Phi\) generated by the vector field
\(F\) defined in \eqref{def:F(mu)} is locally repulsive near \(\nu\) in the
direction \(h\). More precisely, there exist \(b>0\) and a neighbourhood of
\(\nu\) on which
\[
\pi(\mu)h-\mu h\geq b\,\mu h.
\]
Thus, as long as the process remains near \(\nu\), the deterministic drift tends to increase the unstable coordinate \(\mu h\) multiplicatively.

The latter corresponds to the hyperbolic mechanism underlying classical non-convergence results
in stochastic approximation. In the results of
\cite{pemantle1990nonconvergence}, \cite{tarres2000pieges}, \cite{raimond2023nonconvergenceunstableequilibriumscontinuoustime} and \cite{Ben96}, this instability is combined with a non-degeneracy assumption on the noise. A related technique appears in \cite{pemantle1992vertex} for
vertex-reinforced random walks, where convergence to certain equilibria is
excluded by exploiting unstable directions of the associated deterministic
dynamics.

The stochastic fluctuations in the present reinforced Markovian setting are,
however, degenerate near \(\nu\). Projecting the dynamics onto \(h\) and using
the Poisson equation decomposition, the conditional variance of the relevant
martingale increment is controlled by a quantity of the form
\[
h(X_n)+\mu_nh.
\]
The degeneracy therefore depends not only on the unstable coordinate
\(\mu_nh\), but also on the current state \(X_n\). This makes the argument close
in spirit to the degenerate-noise mechanisms of
\cite{pages2004can,tarres2001pieges}, while preventing a direct reduction to
the standard two-armed bandit estimates.

The proof proceeds in four steps. We first establish the deterministic
instability near \(\nu\) and the corresponding variance estimate. We then
introduce deterministic block times \((T_n)\) and use the weak reinforcement
condition to prove
\[
\limsup_{n\to\infty}
\frac{\mu_{T_n}h}{\gamma_{T_n}}
=
+\infty
\qquad\text{almost surely}.
\]
Next, a Poisson equation decomposition yields a block concentration estimate
for the growth of \(\mu_{T_n}h\). Finally, these two estimates imply that the
process cannot remain indefinitely in a neighbourhood of \(\nu\).

%================================================
\subsection{Proof of Theorem~\ref{th:weak-regime}}
%================================================

Fix \(\alpha\in\{1,\ldots,\kappa-1\}\) and set 
\[
\nu:=\nu_\alpha.\] 
Let \(h\in\mathbb R_+^S\setminus\{0\}\) satisfy
\[
Ph=\rho_\kappa h.
\]
Such a vector exists by Lemma~\ref{lem:right-perron-support}.
Fix \(n_0\geq1\) such that for $n\geq n_0$, \(\gamma_n \leq 1/2\) and \(\gamma_{n+1}\leq\gamma_n\).
All block times introduced below are chosen after \(n_0\).

Set \(U:=\supp(h)\). Lemma~\ref{lem:right-perron-support} shows that
\(U^c\) is closed, and hence \(P(x,U)=0\) for \(x\in U^c\). Since
\(\nu P=\rho_\alpha\nu\) and \(\rho_\alpha<\rho_\kappa\), the same lemma gives \(\nu h=0\) and \(\nu(U)=0.\)

\subsubsection{Some deterministic estimates}

This first lemma quantifies the instability of the equilibrium \(\nu\) for the
underlying deterministic dynamics.

\begin{lemma}
\label{lemma:deterministic}
There exist \(b>0\) and a neighbourhood \(\mathcal V_1\) of \(\nu\) in \(\Delta\) such that,
for every \(\mu\in\mathcal V_1\),
\[
\pi(\mu)h\geq (1+b)\mu h.
\]
\end{lemma}

\begin{proof}
Since \(Gh=(1-\rho_\kappa)^{-1}h\), we have, for every \(\mu\in\Delta\),
\[
\pi(\mu)h
=
\frac{\mu Gh}{\mu G\mathbf 1}
=
\frac{1}{(1-\rho_\kappa)\mu G\mathbf 1}\,\mu h.
\]
Moreover, since \(\nu P=\rho_\alpha\nu\), one has
\(\nu G=(1-\rho_\alpha)^{-1}\nu\). Therefore
\[
\frac{1}{(1-\rho_\kappa)\nu G\mathbf 1}
=
\frac{1-\rho_\alpha}{1-\rho_\kappa}
=
1+\frac{\rho_\kappa-\rho_\alpha}{1-\rho_\kappa}
>
1.
\]
The result follows by continuity. 
\end{proof}

Following the Poisson-equation decomposition used in the proof of
\cite[Lemma~2.4]{BC2015}, define the rank-one operator
\[
\Pi_\mu f:=\mathbf 1\,\pi(\mu)f,
\qquad
f\in\mathbb R^S,
\]
and set
\[
A_\mu:=I-K_\mu+\Pi_\mu.
\]

We first observe that \(A_\mu\) is invertible for every \(\mu\in\Delta\).
Indeed, if \(A_\mu f=0\), then applying \(\pi(\mu)\) gives
\(\pi(\mu)f=0\). Hence \(\Pi_\mu f=0\) and \(K_\mu f=f\). Since \(K_\mu\)
is indecomposable, choose a state in its unique recurrent class. This state is
accessible from every state, and optional stopping at its hitting time shows
that every \(K_\mu\)-harmonic function is constant. Then
\(\pi(\mu)f=0\) gives \(f=0\).

We may thus define
\[
g_\mu
:=
A_\mu^{-1}(h-\Pi_\mu h).
\]
Applying \(\pi(\mu)\) to
\(A_\mu g_\mu=h-\Pi_\mu h\) gives
\(\pi(\mu)g_\mu=0\). Consequently, \(g_\mu\) is the unique
\(\pi(\mu)\)-centered solution of
\begin{equation}
\label{eq:defpoisson}
(I-K_\mu)g_\mu
=
h-\Pi_\mu h,
\qquad
\pi(\mu)g_\mu=0.
\end{equation}

It remains to record the dependence on \(\mu\). Since
\(\mu G\mathbf 1\geq1\) for every \(\mu\in\Delta\), there exists an open
neighbourhood \(\mathcal W\) of \(\Delta\) in \(E^1\) on which
\(\mu\mapsto\pi(\mu)\), and hence \(\mu\mapsto A_\mu\), is smooth. Define
\[
\mathscr P(\mu,g)
:=
A_\mu g-(h-\Pi_\mu h),
\qquad
(\mu,g)\in\mathcal W\times\mathbb R^S.
\]
Then \(\mathscr P\) is smooth and
\[
D_g\mathscr P(\mu,g_\mu)=A_\mu,
\]
which is invertible. The implicit function theorem, together with uniqueness
of the centered Poisson solution, therefore shows that
\(\mu\mapsto g_\mu\) is smooth on a neighbourhood of \(\Delta\).

For \(x\in S\) and \(\mu\in\Delta\), we now define
\begin{equation}
\label{eq:def-vmu}
v_\mu(x)
:=
K_\mu g_\mu^2(x)-\bigl(K_\mu g_\mu(x)\bigr)^2.
\end{equation}
This is the conditional variance of the Poisson martingale increment
\(g_\mu(X_{n+1})-K_\mu g_\mu(X_n)\), conditionally on \(X_n=x\) and
\(\mu_n=\mu\).

The following lemma provides a quantitative bound on \(v_\mu(x)\) in a
neighbourhood of \(\nu\).

\begin{lemma}
\label{lem:boundvariance}
There exist a neighbourhood \(\mathcal V_2\) of \(\nu\) and a constant \(C >0\) such that, for every \(\mu\in\mathcal V_2\) and every \(x\in S\),
\[
v_\mu(x)\leq C\bigl(h(x)+\mu h\bigr).
\]
\end{lemma}

\begin{proof}
Throughout the proof, \(C>0\) denotes a constant whose value may change from one line to another.
Observing that \(v_\mu(x) \leq K_\mu g_\mu^2(x),\) it is enough to prove that, for \(\mu\) sufficiently close to \(\nu\),
\[
K_\mu g_\mu^2(x)
\leq
C\bigl(h(x)+\mu h\bigr),
\qquad x\in S.
\]
Smoothness of \(\mu\mapsto g_\mu\) on a neighbourhood of \(\Delta\), together with compactness of \(\Delta\), gives
\[
\|g_\mu\|_\infty\leq C,
\qquad
\|g_\mu-g_\eta\|_\infty
\leq
C\|\mu-\eta\|_1,
\qquad
\mu,\eta\in\Delta.
\]

Recall that \(U=\supp(h)\). Hence \(h=0\) on \(U^c\) and \(h>0\) on \(U\).
Since \(S\) is finite, we may set
\[
h_*:=\min_{y\in U}h(y)>0.
\]
For every \(\mu\in\Delta\), one has \(\mu h\geq h_*\mu(U)\), and therefore
\begin{equation}
\label{eq:muU-muh}
\mu(U)\leq h_*^{-1}\mu h.
\end{equation}

\paragraph{Step 1: Control of \(\boldsymbol{g_\mu}\) on \(\boldsymbol{U^c}\).}
Let \(\eta\in\Delta\) satisfy \(\eta(U)=0\). As \(U^c\) is closed for the accessibility relation,
\((\eta G)(U)=0\), hence \(\pi(\eta)(U)=0\) and
\(\pi(\eta)h=0\). Moreover,
\(K_\eta(x,U)=0\) for \(x\in U^c\), so \(U^c\) is invariant under \(K_\eta\) and
\[
(I-K_\eta)g_\eta=0
\qquad\text{on } \ U^c.
\]
Since \(U^c\) is \(K_\eta\)-invariant and \(K_\eta\) is indecomposable, its restriction to \(U^c\) is indecomposable. Therefore, the preceding optional-stopping argument holds and yields that \(g_\eta\) is constant on \(U^c\). Since \(\pi(\eta)\) is supported on
\(U^c\), the centering condition \(\pi(\eta)g_\eta=0\) forces this constant to be zero, in other words
\begin{equation}
\label{eq:geta-zero-Uc}
g_\eta=0
\qquad\text{on } \ U^c.
\end{equation}

We now show that $g_\mu = O(\mu h)$ on $U^c$.  Since \(\nu(U^c)=1\), we may choose a neighbourhood
\(\mathcal V_2\) of \(\nu\) such that \(\mu(U^c)>0\) for every
\(\mu\in\mathcal V_2\). For such \(\mu\), set
\[
\eta_\mu
:=
\frac{\mu(\cdot\cap U^c)}{\mu(U^c)}.
\]
Then \(\eta_\mu\in\Delta\),
\(\eta_\mu(U)=0\), and
\[
\|\mu-\eta_\mu\|_1=2\mu(U).
\]
Applying \eqref{eq:geta-zero-Uc} to \(\eta_\mu\), and then using the Lipschitz
dependence of \(g_\mu\), we obtain
\[
\sup_{y\in U^c}|g_\mu(y)|
=
\sup_{y\in U^c}|g_\mu(y)-g_{\eta_\mu}(y)|
\leq
C\|\mu-\eta_\mu\|_1
=
2C\mu(U).
\]
Together with \eqref{eq:muU-muh}, this gives
\begin{equation}
\label{eq:gmu-Uc-bound}
\sup_{y\in U^c}|g_\mu(y)|
\leq
C\mu h.
\end{equation}

\paragraph{Step 2: Estimate of \(\boldsymbol{K_\mu g_\mu^2}\).}
Suppose first that \(x\in U\). Since \(K_\mu(x,\cdot)\) is a probability
measure and \(g_\mu\) is bounded on \(\mathcal V_2\),
\[
K_\mu g_\mu^2(x)\leq C.
\]
Since \(h(x)\geq h_*\), it follows that
\[
K_\mu g_\mu^2(x)
\leq
Ch(x)
\leq
C\bigl(h(x)+\mu h\bigr).
\]
Suppose now that \(x\in U^c\). Splitting according to whether the next state belongs to \(U\) or \(U^c\), we have
\[
K_\mu g_\mu^2(x)
=
\sum_{y\in U}K_\mu(x,y)g_\mu^2(y)
+
\sum_{y\in U^c}K_\mu(x,y)g_\mu^2(y).
\]
As \(U^c\) is closed, \(P(x,U)=0\) and for \(y\in U\), \(K_\mu(x,y)=q(x)\mu(y)\). Therefore by \eqref{eq:muU-muh},
\[
\sum_{y\in U}K_\mu(x,y)g_\mu^2(y) \leq C\mu(U)\leq C\mu h.
\]

For the second term, \eqref{eq:gmu-Uc-bound} gives
\[
\sum_{y\in U^c}K_\mu(x,y)g_\mu^2(y)
\leq
\sup_{y\in U^c}|g_\mu(y)|^2
\leq
C(\mu h)^2 \leq C \mu h.
\]
Thus, for \(x\in U^c\),
\[
K_\mu g_\mu^2(x)
\leq
C\mu h
=
C\bigl(h(x)+\mu h\bigr).
\]

Combining the two cases yields
\[
v_\mu(x)
\leq
K_\mu g_\mu^2(x)
\leq
C\bigl(h(x)+\mu h\bigr),
\qquad
\mu\in\mathcal V_2,\ x\in S,
\]
which gives the desired quantitative bound for $v_\mu(x)$ on $\mathcal V_2$.
\end{proof}

\subsubsection{A block-time limsup estimate in the unstable direction}
We introduce deterministic block times adapted to the
stochastic-approximation time scale. 

The proof relies on comparing the weights over intervals of bounded cumulative
step-size and on arbitrarily late excursions into the classes
\(S_{\beta(i)}\subset U\) provided by
Lemma~\ref{lem:right-perron-support}\textup{(iii)--(iv)}, for which
\[
\min_{x\in S_{\beta(i)}}h(x)>0.
\]

Fix \(\tau>0\) and a deterministic \(T_0\geq n_0\), where \(n_0\) is chosen so
that \((\gamma_n)_{n\geq n_0}\) is non-increasing. Define recursively
\[
T_{n+1}
:=
\inf\left\{
m>T_n:
\sum_{k=T_n+1}^{m}\gamma_k\geq \tau
\right\}.
\]

\begin{lemma}[Comparison of weights]
\label{lem:controlled-window-weights}
For every \(\Theta>0\), there exist \(c_\Theta>0\) and
\(N_\Theta\geq n_0\) such that, whenever \(N_\Theta\leq m\leq n\) and
\(\sum\limits_{j=m+1}^{n}\gamma_j\leq\Theta\), one has
\[
\frac{w_m}{w_n}\geq c_\Theta.
\]
\end{lemma}

\begin{proof}
Observe that
\(
\frac{w_m}{w_n}
=
\frac{\gamma_m}{\gamma_n}
\prod\limits_{j=m+1}^n(1-\gamma_j).
\)
Since \((\gamma_n)\) is eventually non-increasing, 
\[
\frac{w_m}{w_n}
\geq \prod_{j=m+1}^n(1-\gamma_j) .\] For \(N_\Theta\) sufficiently large, \(\gamma_j\leq1/2\), and therefore
\(1-\gamma_j\geq e^{-2\gamma_j}\). Hence
\[
\frac{w_m}{w_n}
\geq
\exp\left(-2\sum_{j=m+1}^n\gamma_j\right)
\geq e^{-2\Theta}.
\]
\end{proof}

\begin{lemma}[Block-time limsup estimate]
\label{lem:block-time-limsup}
With the notation above,
\[
\limsup_{n\to\infty}
\frac{\mu_{T_n}h}{\gamma_{T_n}}
=
+\infty
\qquad\text{almost surely}.
\]
\end{lemma}

\begin{proof}
Set
\[
R_n:=\frac{\mu_nh}{r_n}.
\]
By the weighted occupation formula,
\[
R_n
=
w_0\mu_0h+\sum_{k=1}^n w_kh(X_k),
\qquad
\frac{\mu_nh}{\gamma_n}
=
\frac{R_n}{w_n}.
\]
In particular, \((R_n)\) is non-decreasing. Since \(U^c\) is closed, Hypothesis~\ref{hyp:initial-support} gives
\(\mu_0(U)>0\) almost surely. Hence the events
\[
E_i:=\{\mu_0(i)>0\},
\qquad i\in U,
\]
cover an almost sure event. Fix \(i\in U\) and work on \(E_i\).

By Lemma~\ref{lem:right-perron-support}\textup{(iii)--(iv)}, there exists an
index \(\beta(i)\) such that \(S_{\beta(i)}\subset U\), the
class \(S_{\beta(i)}\) is accessible from \(i\) along a \(P\)-path contained
in \(U\), and
\[
\rho_{\beta(i)}=\rho_\kappa,
\qquad
h_i^*:=\min_{x\in S_{\beta(i)}}h(x)>0.
\]
Since \(P_{S_{\beta(i)}}\) is irreducible and has positive Perron value, every
state of \(S_{\beta(i)}\) has a successor in \(S_{\beta(i)}\) with positive
transition probability. A path that has entered \(S_{\beta(i)}\) can therefore
be prolonged inside this class for any prescribed finite number of steps.

Fix \(L\geq1\). We may consequently choose
\begin{align}\label{eq:i-path-L-visits}
    i=y_0^{(L)},y_1^{(L)},\ldots,y_{\ell_L}^{(L)}
\end{align}
such that
\begin{align}\label{eq:lambda^L-visits}
\lambda^{(L)}
:=
\prod_{r=0}^{\ell_L-1}
P(y_r^{(L)},y_{r+1}^{(L)})
>0,
\qquad
\sum_{r=1}^{\ell_L}
\mathbf 1_{\{y_r^{(L)}\in S_{\beta(i)}\}}
\geq L.
\end{align}

Now for every \(x\in S\), Hypothesis~\ref{hyp:main} provides an integer \(\ell_x\geq0\) and states
\[
z_0^x=x,z_1^x,\ldots,z_{\ell_x}^x
\]
such that
\[
d_x
:=
\left(
\prod_{r=0}^{\ell_x-1}
P(z_r^x,z_{r+1}^x)
\right)
q(z_{\ell_x}^x)
>0.
\]
By finiteness of \(S\), let
\[
p:=1+\max_{x\in S}\ell_x,
\qquad
d:=\min_{x\in S}d_x>0.
\]

For \(n\geq0\), we next define 
\[
\sigma_n:=n+\ell_{X_n}+1. 
\]
Since \(X_n\) is \(\mathcal F_n\)-measurable, \(\sigma_n\) is an
\((\mathcal F_k)\)-stopping time determined at time \(n\), with
\[
n+1\leq\sigma_n\leq n+p.
\]
Let \(\mathcal I_n\in\mathcal F_{\sigma_n}\) be the event that the process follows the path associated with \(X_n\) and then jumps to \(i\):
\[
\mathcal I_n
=
\left\{
X_{n+r}=z_r^{X_n},\ 1\leq r\leq\ell_{X_n},
\quad
X_{\sigma_n}=i
\right\}.
\]
% Then \(\mathcal I_n\in\mathcal F_{\sigma_n}\). 
Observe that along the prescribed path \(K_\mu\geq P\), while at its endpoint
\[
K_{\mu_{n+\ell_{X_n}}}
\bigl(z_{\ell_{X_n}}^{X_n},i\bigr)
\geq
q(z_{\ell_{X_n}}^{X_n})
\mu_{n+\ell_{X_n}}(i).
\]
Moreover, since \(\ell_{X_n}\leq p-1\) and \((r_n)\) is non-increasing,
\[
\mu_{n+\ell_{X_n}}(i)
\geq
r_{n+\ell_{X_n}}\mu_0(i)
\geq
r_{n+p}\mu_0(i).
\]
Therefore
\[
\mathbb P(\mathcal I_n\mid\mathcal F_n)
\geq
d\,\mu_0(i)r_{n+p}.
\]

We now consider the auxiliary deterministic grid
\[
u_q:=qb_L,
\quad q\geq0, \qquad \text{ where } \qquad b_L:=p+\ell_L
\]
and define \(\mathcal H_q:=\mathcal F_{u_q}\) for \(q\geq0\). The number \(b_L\) is a uniform upper bound on the duration of the prescribed
excursion; it need not be its exact duration.
The choice of \(b_L\) ensures that, irrespective of \(X_{u_q}\),
\[
\sigma_{u_q}+\ell_L
\leq
u_q+p+\ell_L
=
u_{q+1}.
\]
Thus the two prescribed parts of the excursion are completed by time
\(u_{q+1}\).

Let
\[
\mathcal J_q^{(L)}
:=
\left\{
X_{\sigma_{u_q}+r}=y_r^{(L)},
\ 1\leq r\leq\ell_L
\right\},
\]
and set
\[
\mathcal R_q^{(L)}
:=
\mathcal I_{u_q}\cap\mathcal J_q^{(L)}.
\]
On \(\mathcal R_q^{(L)}\), the process first reaches \(i\) through the path
associated with \(X_{u_q}\), and then follows \(y^{(L)}\). Since
\(\sigma_{u_q}+\ell_L\leq u_{q+1}\),
\[
\mathcal R_q^{(L)}\in\mathcal F_{u_{q+1}}.
\]

On \(\mathcal I_{u_q}\), one has \(X_{\sigma_{u_q}}=i\). By successive
conditioning and the inequality \(K_\mu\geq P\),
\[
\mathbb P\left(
\mathcal J_q^{(L)}
\,\middle|\,
\mathcal F_{\sigma_{u_q}}
\right)
\geq
\lambda^{(L)}
\qquad\text{on }\mathcal I_{u_q}.
\]
Since \(\mathcal I_{u_q}\in\mathcal F_{\sigma_{u_q}}\), the tower property
gives
\[
\begin{aligned}
\mathbb P\left(
\mathcal R_q^{(L)}
\,\middle|\,
\mathcal F_{u_q}
\right)
&=
\mathbb E\left[
\mathbf1_{\mathcal I_{u_q}}
\mathbb P\left(
\mathcal J_q^{(L)}
\,\middle|\,
\mathcal F_{\sigma_{u_q}}
\right)
\,\middle|\,
\mathcal F_{u_q}
\right]\\
&\geq
\lambda^{(L)}
\mathbb P\left(
\mathcal I_{u_q}
\,\middle|\,
\mathcal F_{u_q}
\right)\\
&\geq
d\lambda^{(L)}\mu_0(i)r_{u_q+p}.
\end{aligned}
\]
Set
\[
a_L:=d\lambda^{(L)}\mu_0(i).
\]
The random variable \(a_L\) is \(\mathcal F_0\)-measurable and strictly
positive on \(E_i\).
Since \(u_q+p\leq u_{q+1}\),
\[
\mathbb P\left(
\mathcal R_q^{(L)}
\,\middle|\,
\mathcal F_{u_q}
\right)
\geq
a_Lr_{u_{q+1}}
\qquad\text{on }E_i.
\]
The intervals \([u_q,u_{q+1})\) partition \(\mathbb N\), and
\[
\sum_{n=u_q}^{u_{q+1}-1}r_n
\leq
b_Lr_{u_q}.
\]
By the weak reinforcement regime condition \(\sum_n r_n=\infty\),
\[
\sum_qr_{u_q}=\infty,
\qquad
\sum_qr_{u_{q+1}}=\infty.
\]
Thus, on \(E_i\),
\[
\sum_{q\geq0}
\mathbb P\left(
\mathcal R_q^{(L)}
\,\middle|\,
\mathcal F_{u_q}
\right)
=
\infty.
\]
Since
\(\mathcal R_q^{(L)}\in\mathcal H_{q+1}\),
Lévy's conditional Borel--Cantelli lemma gives
\[
\mathbb P\left(
E_i\cap
\{\mathcal R_q^{(L)}\text{ infinitely often}\}
\right)
=
\mathbb P(E_i).
\]

On \(\mathcal R_q^{(L)}\), let
\[
\Lambda_q^{(L)}
:=
\left\{
\sigma_{u_q}+r:
1\leq r\leq\ell_L,\
y_r^{(L)}\in S_{\beta(i)}
\right\},
\]
such that on \(\mathcal R_q^{(L)}\),
\[
|\Lambda_q^{(L)}|\geq L,
\qquad
h(X_t)\geq h_i^*
\quad\text{for }t\in\Lambda_q^{(L)}.
\]

Let
\[
N_q:=\inf\{n:T_n\geq u_{q+1}\}.
\]
Then \(N_q\to\infty\), and, for all sufficiently large \(q\),
\[
T_{N_q-1}<u_{q+1}\leq T_{N_q}.
\]
For \(t\in\{u_q,\ldots,u_{q+1}\}\),
\[
\sum_{j=t+1}^{T_{N_q}}\gamma_j
\leq
\sum_{j=u_q+1}^{u_{q+1}}\gamma_j
+
\sum_{j=u_{q+1}+1}^{T_{N_q}}\gamma_j.
\]
For all sufficiently large \(q\),
\[
\sum_{j=u_q+1}^{u_{q+1}}\gamma_j
\leq
b_L\gamma_{u_q+1}
\leq1,
\]
while
\[
\sum_{j=u_{q+1}+1}^{T_{N_q}}\gamma_j
<
\tau+\gamma_{T_{N_q}}
\leq\tau+1.
\]
Therefore
\[
\sum_{j=t+1}^{T_{N_q}}\gamma_j\leq\tau+2.
\]
For all sufficiently large \(q\), with
\(c_\tau:=c_{\tau+2}\), Lemma~\ref{lem:controlled-window-weights} gives
\[
\frac{w_t}{w_{T_{N_q}}}
\geq
c_\tau,
\qquad
u_q\leq t\leq u_{q+1}.
\]

Notice that the prescribed excursion need not occupy the whole interval \([u_q,u_{q+1}]\). If it is completed earlier, its contributions have already
been added to the non-decreasing quantity \(R_n\). The preceding comparison shows that their weights remain uniformly comparable to
\(w_{T_{N_q}}\).

Consequently, on \(\mathcal R_q^{(L)}\),
\[
\frac{\mu_{T_{N_q}}h}{\gamma_{T_{N_q}}}
=
\frac{R_{T_{N_q}}}{w_{T_{N_q}}}
\geq
\sum_{t\in\Lambda_q^{(L)}}
\frac{w_t}{w_{T_{N_q}}}h(X_t)
\geq
c_\tau Lh_i^*.
\]
Since \(\mathcal R_q^{(L)}\) occurs infinitely often almost surely on \(E_i\)
and \(N_q\to\infty\), it follows that, for this fixed \(L\),
\[
\limsup_{n\to\infty}
\frac{\mu_{T_n}h}{\gamma_{T_n}}
\geq
c_\tau Lh_i^*
\qquad\text{almost surely on }E_i.
\]

Taking the intersection over \(L\in\mathbb N^*\), and recalling that \(c_\tau h_i^*>0\) is independent of \(L\), we obtain
\[
\limsup_{n\to\infty}
\frac{\mu_{T_n}h}{\gamma_{T_n}}
=
+\infty
\qquad\text{almost surely on }E_i.
\]
Finally, since \(U\) is finite and the events \(E_i\), \(i\in U\), cover an almost sure event, the result follows.

\end{proof}

\subsubsection{Block decomposition and concentration estimate}

We retain the deterministic times \((T_n)\) defined in the preceding subsection. For \(T_n\leq k\leq T_{n+1}\), the recursion for \((\mu_n)\) and the
non-negativity of \(h\) give
\begin{equation}
\label{eq:deterministlowerbound}
\mu_kh
\geq
\frac{r_k}{r_{T_n}}\mu_{T_n}h.
\end{equation}
Moreover,
\(
\frac{r_{T_{n+1}}}{r_{T_n}}
=
\prod\limits_{k=T_n+1}^{T_{n+1}}(1-\gamma_k).
\)
Since
\[
\sum_{k=T_n+1}^{T_{n+1}}\gamma_k
<
\tau+\gamma_{T_0+1},
\]
the same product estimate as in
Lemma~\ref{lem:controlled-window-weights} gives a constant \(c>0\) such that
\begin{equation}
\label{eq:minorationquotient}
\inf_{n\geq0}
\frac{r_{T_{n+1}}}{r_{T_n}}
\geq c.
\end{equation}

We set
\[
\delta:=\frac{bc\tau}{2},
\qquad
\mathcal G_n:=\mathcal F_{T_n},
\qquad
Y_n:=\mu_{T_n}h,
\qquad
\mathcal V:=\mathcal V_1\cap\mathcal V_2,
\]
where \(\mathcal V_1\) and \(\mathcal V_2\) are given by
Lemma~\ref{lemma:deterministic} and Lemma~\ref{lem:boundvariance}. By the
positivity of \(\mu_0(U)\) recalled above and the weighted representation
\eqref{eq:expressionexplicitemu}, one has \(Y_n>0\) almost surely.

Let
\begin{equation}
\label{eq:definitioneventAn}
\mathcal A_{n+1}
:=
\left\{
\mu_k\in\mathcal V
\text{ for all }T_n\leq k<T_{n+1}
\right\}.
\end{equation}
Then \(\mathcal A_{n+1}\in\mathcal G_{n+1}\). We also introduce the event
corresponding to geometric growth of the unstable coordinate over one block
\[
\mathcal C_{n+1}
:=
\left\{
Y_{n+1}\geq(1+\delta)Y_n
\right\}.
\]
On \(\mathcal A_{n+1}\), the deterministic instability favours this geometric
growth. The next proposition shows that the bias and martingale terms prevent this growth only with probability of order \(\gamma_{T_n}/Y_n\).

\begin{proposition}[Block concentration estimate]
\label{prop:concentration}
There exists a constant \(C>0\) such that, for every \(n\geq0\),
\[
\mathbb P\left(
\mathcal C_{n+1}^c\cap\mathcal A_{n+1}
\,\middle|\,
\mathcal G_n
\right)
\leq
C\frac{\gamma_{T_n}}{Y_n}.
\]
\end{proposition}

\begin{proof}
Throughout the proof, \(C>0\) denotes a constant whose value may change from one line to another.
The proof is based on a decomposition of the increment \(Y_{n+1}-Y_n\) into
three contributions: a deterministic drift term, a martingale term, and a bias
term.

\paragraph{Step 1: Drift--martingale--bias decomposition.}
For every \(j\geq0\), the recursion defining \((\mu_n)\) gives
\[
\mu_{j+1}h
=
\mu_jh
+
\gamma_{j+1}\bigl(\pi(\mu_j)h-\mu_jh\bigr)
+
\gamma_{j+1}\bigl(h(X_{j+1})-\pi(\mu_j)h\bigr).
\]
Set
\[
m_{j+1}
:=
g_{\mu_j}(X_{j+1})-K_{\mu_j}g_{\mu_j}(X_j).
\]
Using the Poisson equation \eqref{eq:defpoisson}, we write
\[
\begin{aligned}
h(X_{j+1})-\pi(\mu_j)h
&=
m_{j+1}
+
K_{\mu_j}g_{\mu_j}(X_j)
-
K_{\mu_j}g_{\mu_j}(X_{j+1}).
\end{aligned}
\]
The variables \(m_{j+1}\) are martingale increments, since
\(\mathbb E[m_{j+1}\mid\mathcal F_j]=0\).

For readability, we write \(H_\mu(x):=K_\mu g_\mu(x)\). Summing over
\(j=T_n,\ldots,T_{n+1}-1\), we obtain
\[
Y_{n+1}
=
Y_n+D_{n+1}+M_{n+1}+B_{n+1},
\]
where
\[
D_{n+1}:=
\sum_{j=T_n}^{T_{n+1}-1}
\gamma_{j+1}\bigl(\pi(\mu_j)h-\mu_jh\bigr),
\qquad
M_{n+1}:=
\sum_{j=T_n}^{T_{n+1}-1}
\gamma_{j+1}m_{j+1},
\]
and
\[
B_{n+1}:=
\sum_{j=T_n}^{T_{n+1}-1}
\bigl(b_j^{(1)}+b_j^{(2)}+b_j^{(3)}\bigr),
\]
with
\[
\begin{aligned}
b_j^{(1)}
&:=
(\gamma_{j+1}-\gamma_j)H_{\mu_j}(X_j),\\
b_j^{(2)}
&:=
\gamma_jH_{\mu_j}(X_j)
-
\gamma_{j+1}H_{\mu_{j+1}}(X_{j+1}),\\
b_j^{(3)}
&:=
\gamma_{j+1}
\bigl(H_{\mu_{j+1}}(X_{j+1})-H_{\mu_j}(X_{j+1})\bigr).
\end{aligned}
\]
This mirrors the Poisson equation decomposition in~\cite[Lemma~2.4]{BC2015}: \(D_{n+1}\) is the
drift, \(M_{n+1}\) is the martingale fluctuation, and \(B_{n+1}\) collects the
bias terms.

\paragraph{Step 2: Deterministic growth on \(\mathcal A_{n+1}\).}
On \(\mathcal A_{n+1}\), all measures \(\mu_j\), \(T_n\leq j<T_{n+1}\), belong
to \(\mathcal V\). Hence Lemma~\ref{lemma:deterministic} gives
\[
\pi(\mu_j)h-\mu_jh\geq b\,\mu_jh.
\]
Using the respective lower bounds in \eqref{eq:deterministlowerbound} and \eqref{eq:minorationquotient}, we
also have \(\mu_jh\geq cY_n\) for \(T_n\leq j<T_{n+1}\). Therefore, by the
choice of \(\delta\),
\[
D_{n+1}
\geq
bcY_n\sum_{j=T_n}^{T_{n+1}-1}\gamma_{j+1}
\geq
bc\tau Y_n
=
2\delta Y_n.
\]

We now introduce
\[
\mathcal L_{n+1}
:=
\left\{
\mu_kh\leq \frac{1+\delta}{c}Y_n
\text{ for all }T_n\leq k<T_{n+1}
\right\}.
\]
If \(\mathcal L_{n+1}^c\) occurs, then for some
\(k\in\{T_n,\ldots,T_{n+1}-1\}\), one has
\(\mu_kh>(1+\delta)Y_n/c\). The same deterministic lower bound as in \eqref{eq:deterministlowerbound}, applied between times \(k\) and \(T_{n+1}\), gives
\[
Y_{n+1}
=
\mu_{T_{n+1}}h
\geq
\frac{r_{T_{n+1}}}{r_k}\mu_kh.
\]
Since \(k\geq T_n\), we have
\(r_{T_{n+1}}/r_k\geq r_{T_{n+1}}/r_{T_n}\geq c\). Hence
\[
Y_{n+1}\geq c\,\mu_kh>(1+\delta)Y_n.
\]
Thus \[
\mathcal L_{n+1}^c\subset\mathcal C_{n+1}.
\]
On
\(\mathcal A_{n+1}\cap\mathcal L_{n+1}\cap\mathcal C_{n+1}^c\),
the decomposition above and the inequality \(D_{n+1}\geq2\delta Y_n\)
imply
\[
M_{n+1}+B_{n+1}<-\delta Y_n.
\]
Consequently,
\begin{equation}
\label{eq:probasplit}
\begin{aligned}
\mathbb P\left(
\mathcal C_{n+1}^c\cap\mathcal A_{n+1}
\,\middle|\,
\mathcal G_n
\right)
&\leq
\mathbb P\left(
B_{n+1}<-\frac{\delta}{2}Y_n
\,\middle|\,
\mathcal G_n
\right)\\
&\quad+
\mathbb P\left(
M_{n+1}<-\frac{\delta}{2}Y_n,\,
\mathcal A_{n+1}\cap\mathcal L_{n+1}
\,\middle|\,
\mathcal G_n
\right).
\end{aligned}
\end{equation}

\paragraph{Step 3: Control of the bias term.}
Since \(S\) is finite and \(\mu\mapsto g_\mu\) is smooth, the map
\((\mu,x)\mapsto H_\mu(x)\) is bounded and Lipschitz in \(\mu\). Hence, for a
deterministic constant \(C>0\),
\[
|b_j^{(1)}|
\leq
C|\gamma_{j+1}-\gamma_j|,
\qquad
\left|
\sum_{j=T_n}^{T_{n+1}-1}b_j^{(2)}
\right|
\leq
C\gamma_{T_n},
\qquad
|b_j^{(3)}|
\leq
C\gamma_{j+1}^2.
\]
The second estimate follows from the telescopic structure of \(b_j^{(2)}\);
the third one uses
\(\|\mu_{j+1}-\mu_j\|_1\leq 2\gamma_{j+1}\). Since \((\gamma_n)\) is
non-increasing on the blocks and the block ODE-time length is bounded,
\[
\sum_{j=T_n}^{T_{n+1}-1}|\gamma_{j+1}-\gamma_j|\leq\gamma_{T_n},
\qquad
\sum_{j=T_n}^{T_{n+1}-1}\gamma_{j+1}^2\leq C\gamma_{T_n}.
\]
Combining these estimates gives 
\[
|B_{n+1}|\leq C\gamma_{T_n}.
\]
It follows that
\[
\mathbb P\left(
B_{n+1}<-\frac{\delta}{2}Y_n
\,\middle|\,
\mathcal G_n
\right)
\leq
\mathbf 1_{\{Y_n\leq C\gamma_{T_n}\}}
\leq
C\frac{\gamma_{T_n}}{Y_n}.
\]

\paragraph{Step 4: Control of the martingale term.}
Let
\[
N_n:=T_{n+1}-T_n.
\]
Define the block martingale
\[
\mathcal M_0^{(n)}:=0,
\qquad
\mathcal M_k^{(n)}
:=
\sum_{j=T_n}^{T_n+k-1}\gamma_{j+1}m_{j+1},
\qquad k\geq1,
\]
such that
\[
M_{n+1}=\mathcal M_{N_n}^{(n)}.
\]
Then \(\mathcal M^{(n)}\) is a square-integrable martingale with predictable
quadratic variation
\[
\langle\mathcal M^{(n)}\rangle_k
=
\sum_{j=T_n}^{T_n+k-1}
\gamma_{j+1}^2v_{\mu_j}(X_j).
\]
Let us introduce the stopping time
\[
\sigma^{(n)}
:=
\inf\left\{
k\geq0:
\mu_{T_n+k}\notin\mathcal V
\text{ or }
\mu_{T_n+k}h>
\frac{1+\delta}{c}Y_n
\right\}.
\]
By definition of \(\mathcal A_{n+1}\) and \(\mathcal L_{n+1}\),
\[
\mathcal A_{n+1}\cap\mathcal L_{n+1}
\subset
\{\sigma^{(n)}\geq N_n\}.
\]
Therefore, by optional stopping and the quadratic variation identity,
\[
\begin{aligned}
\mathbb E\left[
M_{n+1}^2
\mathbf 1_{\mathcal A_{n+1}\cap\mathcal L_{n+1}}
\,\middle|\,
\mathcal G_n
\right]
&\leq
\mathbb E\left[
\bigl(\mathcal M_{N_n\wedge\sigma^{(n)}}^{(n)}\bigr)^2
\,\middle|\,
\mathcal G_n
\right] \\
&=
\mathbb E\left[
\langle\mathcal M^{(n)}\rangle_{N_n\wedge\sigma^{(n)}}
\,\middle|\,
\mathcal G_n
\right] \\
&\leq
\sum_{j=T_n}^{T_{n+1}-1}
\gamma_{j+1}^2
\mathbb E\left[
\mathbf 1_{\{\sigma^{(n)}>j-T_n\}}
v_{\mu_j}(X_j)
\,\middle|\,
\mathcal G_n
\right].
\end{aligned}
\]
On \(\{\sigma^{(n)}>j-T_n\}\), one has \(\mu_j\in\mathcal V\). Hence
Lemma~\ref{lem:boundvariance} gives
\[
v_{\mu_j}(X_j)
\leq
C\bigl(h(X_j)+\mu_jh\bigr).
\]
We first control the contribution of \(\mu_jh\). On
\(\{\sigma^{(n)}>j-T_n\}\), the stopping rule gives
\(\mu_jh\leq (1+\delta)c^{-1}Y_n\). Hence
\[
\sum_{j=T_n}^{T_{n+1}-1}
\gamma_{j+1}^2
\mathbb E\left[
\mathbf 1_{\{\sigma^{(n)}>j-T_n\}}\mu_jh
\,\middle|\,
\mathcal G_n
\right]
\leq
C Y_n
\sum_{j=T_n}^{T_{n+1}-1}\gamma_{j+1}^2.
\]
Moreover,
\[
\sum_{j=T_n}^{T_{n+1}-1}\gamma_{j+1}^2
\leq
\gamma_{T_n+1}
\sum_{j=T_n}^{T_{n+1}-1}\gamma_{j+1}
\leq
C\gamma_{T_n},
\]
because \((\gamma_n)\) is non-increasing on the block and the block has bounded
stochastic-approximation time length. Therefore this contribution is bounded by
\(C\gamma_{T_n}Y_n\).

It remains to bound the contribution of \(h(X_j)\). For \(T_n\leq j\leq
T_{n+1}\), set
\[
u_j
:=
\mathbb E\left[
h(X_j)\mathbf 1_{\{\sigma^{(n)}>j-T_n\}}
\,\middle|\,
\mathcal G_n
\right].
\]
For \(T_n\leq j<T_{n+1}\), the inclusion
\(\{\sigma^{(n)}>j+1-T_n\}\subset\{\sigma^{(n)}>j-T_n\}\), together with
\(K_{\mu_j}h=Ph+q\,\mu_jh\) and \(Ph=\rho_\kappa h\), gives
\[
u_{j+1}
\leq
\mathbb E\left[
\mathbf 1_{\{\sigma^{(n)}>j-T_n\}}
\bigl(\rho_\kappa h(X_j)+q(X_j)\mu_jh\bigr)
\,\middle|\,
\mathcal G_n
\right].
\]
On \(\{\sigma^{(n)}>j-T_n\}\), the stopping rule gives
\(\mu_jh\leq (1+\delta)c^{-1}Y_n\). Therefore
\[
u_{j+1}\leq \rho_\kappa u_j+CY_n.
\]
Moreover, since \(\mu_{T_n}(X_{T_n})\geq\gamma_{T_n}\) and
\(Y_n=\mu_{T_n}h\),
\[
u_{T_n}\leq h(X_{T_n})\leq \frac{Y_n}{\gamma_{T_n}}.
\]
Iterating the previous inequality yields
\[
u_j
\leq
\rho_\kappa^{j-T_n}\frac{Y_n}{\gamma_{T_n}}
+
CY_n,
\qquad T_n\leq j\leq T_{n+1}.
\]
Consequently,
\[
\begin{aligned}
\sum_{j=T_n}^{T_{n+1}-1}\gamma_{j+1}^2u_j
&\leq
C\gamma_{T_n}^2
\sum_{j=T_n}^{T_{n+1}-1}
\rho_\kappa^{j-T_n}
\frac{Y_n}{\gamma_{T_n}}
+
CY_n
\sum_{j=T_n}^{T_{n+1}-1}\gamma_{j+1}^2  \\
&\leq
C\gamma_{T_n}Y_n.
\end{aligned}
\]
Combining the two contributions,
\[
\mathbb E\left[
M_{n+1}^2
\mathbf 1_{\mathcal A_{n+1}\cap\mathcal L_{n+1}}
\,\middle|\,
\mathcal G_n
\right]
\leq
C\gamma_{T_n}Y_n.
\]
Chebyshev's inequality therefore gives
\[
\mathbb P\left(
M_{n+1}<-\frac{\delta}{2}Y_n,\,
\mathcal A_{n+1}\cap\mathcal L_{n+1}
\,\middle|\,
\mathcal G_n
\right)
\leq
C\frac{\gamma_{T_n}}{Y_n}.
\]
Together with \eqref{eq:probasplit} and the bias estimate, this proves the proposition.
\end{proof}

\subsubsection{Escape from a neighbourhood of a lower QSD}
Set
\[
Z_n:=\frac{\gamma_{T_n}}{Y_n}.
\]
By Lemma~\ref{lem:block-time-limsup},
\[
\liminf_{n\to\infty}Z_n=0
\qquad\text{almost surely}.
\]
For \(n\geq0\), let
\[
\mathcal E_n
:=
\bigcap_{\ell\geq 0}\mathcal A_{n+\ell+1},
\]
represent the event that the process remains in \(\mathcal V\) from block \(n\) onward.

\begin{lemma}
\label{lem:escape-small-Z}
There exists a deterministic constant \(C_0>0\) such that, for every
\(n\geq0\),
\[
\mathbb P(\mathcal E_n\mid\mathcal G_n)
\leq
C_0 Z_n.
\]
\end{lemma}

\begin{proof}
For \(\ell\geq0\), set
\[
\mathcal B_\ell^{(n)}
:=
\bigcap_{r=0}^{\ell-1}
\left(
\mathcal A_{n+r+1}\cap\mathcal C_{n+r+1}
\right),
\]
with the convention \(\mathcal B_0^{(n)}=\Omega\). On \(\mathcal B_\ell^{(n)}\),
\(Y_{n+\ell}\geq(1+\delta)^\ell Y_n\). Since
\(\gamma_{T_{n+\ell}}\leq\gamma_{T_n}\), it follows that
\[
Z_{n+\ell}\leq(1+\delta)^{-\ell}Z_n.
\]
Moreover, on \(\mathcal E_n\), at least one of the events
\(\mathcal C_{n+1},\mathcal C_{n+2},\ldots\) must fail, since otherwise
\(Y_{n+\ell}\) would grow geometrically while remaining bounded by
\(\|h\|_\infty\). Hence
\[
\mathcal E_n
\subset
\bigcup_{\ell\geq0}
\left(
\mathcal B_\ell^{(n)}
\cap\mathcal A_{n+\ell+1}
\cap\mathcal C_{n+\ell+1}^c
\right).
\]

Using Proposition~\ref{prop:concentration}, for every \(m\geq0\),
\[
\mathbb P\left(
\mathcal C_{m+1}^c\cap\mathcal A_{m+1}
\,\middle|\,
\mathcal G_m
\right)
\leq
C Z_m.
\]
Moreover, \(\mathcal B_\ell^{(n)}\in\mathcal G_{n+\ell}\). Hence
\[
\begin{aligned}
\mathbb P(\mathcal E_n\mid\mathcal G_n)
&\leq
\sum_{\ell\geq0}
\mathbb E\left[
\mathbf 1_{\mathcal B_\ell^{(n)}}
\mathbb P\left(
\mathcal C_{n+\ell+1}^c\cap\mathcal A_{n+\ell+1}
\,\middle|\,
\mathcal G_{n+\ell}
\right)
\,\middle|\,
\mathcal G_n
\right] \\
&\leq
C\sum_{\ell\geq0}
\mathbb E\left[
\mathbf 1_{\mathcal B_\ell^{(n)}} Z_{n+\ell}
\,\middle|\,
\mathcal G_n
\right] \\
&\leq
C Z_n\sum_{\ell\geq0}(1+\delta)^{-\ell} \leq C_0Z_n,
\end{aligned}
\]
which proves the required upper-bound in Lemma~\ref{lem:escape-small-Z}.
\end{proof}

We now exclude convergence to \(\nu\). If \(\mu_n\to\nu\), then the process
eventually remains in \(\mathcal V\). Hence
\[
\left\{\lim_{n\to\infty}\mu_n=\nu\right\}
\subset
\bigcup_{m\geq0}\mathcal E_m.
\]
Fix \(m\geq0\). Since \(\mathcal E_m\in\mathcal G_\infty\), Lévy's upward theorem gives
\[
\mathbb P(\mathcal E_m\mid\mathcal G_n)
\longrightarrow
\mathbf 1_{\mathcal E_m}
\qquad\text{almost surely}.
\]
For \(n\geq m\), since \(\mathcal E_m\subset\mathcal E_n\),
\[
\mathbb P(\mathcal E_m\mid\mathcal G_n)
\leq
\mathbb P(\mathcal E_n\mid\mathcal G_n)
\leq
C_0Z_n.
\]
Taking an almost sure subsequence along which \(Z_n\to0\) gives
\(\mathbf 1_{\mathcal E_m}=0\). Thus
\(\mathbb P(\mathcal E_m)=0\), and consequently
\[
\mathbb P\Bigl(\lim_{n\to\infty}\mu_n=\nu\Bigr)=0.
\]
\begin{proof}[Proof of Theorem~\ref{th:weak-regime}]
Since \(\nu\neq\nu_\kappa\) was arbitrary, every lower QSD is excluded. By
Theorem~\ref{th:main}, the almost sure limit of \((\mu_n)\) is one of
\(\nu_1,\ldots,\nu_\kappa\). Therefore
\[
\mathbb P\Bigl(\lim_{n\to\infty}\mu_n=\nu_\kappa\Bigr)=1.
\]

\end{proof}

\section{Selection principle in the strong reinforcement regime}\label{sec:strong-regime}

Throughout this section, we work under Hypotheses~\ref{hyp:main},
\ref{hyp:main-two}, \ref{hyp:initial-support}, and \ref{hyp:gamma}, and assume
that the process is in the strong reinforcement regime for which condition \eqref{eq:sum-rn-finite} holds, namely
\begin{equation*}
\sum_{n\geq0}\frac1{W_n}
=
\sum_{n\geq0}r_n
<\infty.
\end{equation*}
The aim of this section is to prove Theorem~\ref{thm:strong-regime}. In the strong reinforcement regime, the process may become trapped in a closed subset of the state space with positive probability. We use this mechanism to show
that every QSD associated with a maximal communicating class is selected with positive probability.

\subsection{Overview of the argument}

Let \(C\subset S\) be closed. On the event that the process remains in \(C\)
after time \(N\),
\[
\mu_n(S\setminus C)
=
\mu_N(S\setminus C)\frac{r_n}{r_N}.
\]
Since \(\sum_n r_n<\infty\), this gives a positive probability of never
leaving \(C\).
We apply this mechanism to the closed sets
\[
C_\alpha:=\overline{S_\alpha},
\qquad
\alpha\in\{1,\ldots,\kappa\}.
\]
After the
process reaches \(S_\alpha\), we couple it with the normalized reinforced
dynamics on \(C_\alpha\). On the event that the coupling never breaks, the two
occupation measures are asymptotically identical. Theorem~\ref{th:main},
applied to the reduced process, then gives positive-probability convergence
to \(\nu_\alpha\).
% The first step is to show that the process can reach a prescribed state
% \(j_\alpha\in S_\alpha\) with positive probability. Once this has happened, we
% restart the construction at the hitting time of \(j_\alpha\) and compare the
% original process with a reduced reinforced process evolving only on
% \(C_\alpha\).

% The comparison is made through a coupling. On the event that the coupling never
% breaks, the original chain remains in \(C_\alpha\) forever, and its occupation measure is asymptotically equivalent to the occupation measure of the reduced
% process. The reduced dynamics satisfies the same general assumptions on the
% closed state space \(C_\alpha\). Moreover, \(S_\alpha\) is the unique class of largest Perron value for the restricted
% kernel, and the associated
% QSD is precisely \(\nu_\alpha\). Therefore Theorem~\ref{th:main}, applied to
% the reduced process, gives convergence to \(\nu_\alpha\) with positive
% probability. Combining this with the positive probability of the good coupling
% event proves the selection statement of Theorem~\ref{thm:strong-regime}.

\subsection{Proof of Theorem~\ref{thm:strong-regime}}

\subsubsection{Coupling with the reduced dynamics}\label{sec:coupling}

Let \(C\subset S\) be non-empty and closed. We couple the original process with a reduced reinforced process on \(C\), whose initial measure is the normalized restriction of \(\mu_0\) to \(C\). An auxiliary process \(B_n\) records whether the two processes remain coupled. On the event that \(B_n=1\) for all \(n\), the original trajectory stays in \(C\) and the two occupation measures become
asymptotically identical.

Recall that the face of the simplex supported on \(C\) is
\[
\Delta^C = \{\mu\in\Delta:\ \supp(\mu)\subset C\} = \{\mu\in\Delta:\ \mu(S\setminus C)=0\}.
\]

\begin{lemma}
\label{lem:coupling-reduced}
Assume that
\[
\sum_{n\geq0}r_n<\infty.
\]
Let \(C\subset S\) be a non-empty closed set. Let \(\mathcal H\) be an initial
sigma-field, and let \(X_0\) and \(\mu_0\) be \(\mathcal H\)-measurable random
variables such that
\[
X_0\in C,
\qquad
\mu_0\in\Delta,
\qquad
\mu_0(C)>0
\qquad\text{almost surely }.
\]
Set
\[
\widetilde\mu_0:=\frac{\mu_0(\cdot\cap C)}{\mu_0(C)}\in\Delta^C,
\qquad
m_0:=\mu_0(C),
\qquad
m_n:=1-r_n(1-m_0),
\]
and
\[
\widetilde\gamma_{n+1}:=\frac{\gamma_{n+1}}{m_{n+1}},
\qquad n\geq0.
\]
Possibly after enlarging the underlying probability space, there exists a process
\[
(X_n,\widetilde X_n,B_n,\mu_n,\widetilde\mu_n)_{n\geq0}
\]
with values in
\[
S\times C\times\{0,1\}\times\Delta\times\Delta^C,
\]
adapted to the filtration
\[
\mathcal F_n
:=
\mathcal H\vee
\sigma\{(X_k,\widetilde X_k,B_k):0\leq k\leq n\},
\]
such that \(\widetilde X_0=X_0\), \(B_0=1\), and the following properties hold:
\begin{enumerate}[label=\textup{(\roman*)}]
\item for every \(n\geq0\) and \(i\in S\),
\[
\mathbb P(X_{n+1}=i\mid\mathcal F_n)=K_{\mu_n}(X_n,i),
\]
where
\[
\mu_{n+1}=(1-\gamma_{n+1})\mu_n+\gamma_{n+1}\delta_{X_{n+1}};
\] 
\item
for every \(n\geq0\) and \(j\in C\),
\[
\mathbb P(\widetilde X_{n+1}=j\mid\mathcal F_n)
=
K_{\widetilde\mu_n}(\widetilde X_n,j),
\]
where
\[
\widetilde\mu_{n+1}
=
(1-\widetilde\gamma_{n+1})\widetilde\mu_n
+
\widetilde\gamma_{n+1}\delta_{\widetilde X_{n+1}}.
\]

\item Given \(\mathcal H\), the process \((B_n)_{n\geq0}\) is a time-inhomogeneous Markov chain on \(\{0,1\}\), independent of \((\widetilde X_n)_{n\geq0}\), with transition matrix
\[
\begin{pmatrix}
1&0\\
1-m_n&m_n
\end{pmatrix},
\]
where rows and columns are indexed in the order \(0,1\). Moreover, for
\[
\mathcal E_C:=\{B_n=1\ \forall n\geq0\},
\]
one has
\[
\mathbb P(\mathcal E_C\mid\mathcal H)
=
\prod_{n\geq0}m_n
>
0
\qquad\text{almost surely }.
\]

\item
For every \(n\geq0\), on \(\{B_n=1\}\),
\[
X_n=\widetilde X_n\in C,
\qquad
\mu_n(\cdot\cap C)=m_n\widetilde\mu_n(\cdot).
\]
In particular, on \(\mathcal E_C\),
\[
X_n=\widetilde X_n
\quad\forall n\geq0,
\qquad
\|\mu_n-\widetilde\mu_n\|_1
=
2(1-m_n)
\longrightarrow0.
\]
\end{enumerate}
\end{lemma}

\begin{proof}
Since \(0<r_n\leq1\) and \(r_n\downarrow0\),
\[
m_n=1-r_n(1-m_0)\in[m_0,1],
\qquad
m_n\uparrow1.
\]
Moreover,
\[
m_{n+1}
=
1-(1-\gamma_{n+1})r_n(1-m_0)
=
(1-\gamma_{n+1})m_n+\gamma_{n+1}.
\]
In particular, \(m_{n+1}-\gamma_{n+1}
=
(1-\gamma_{n+1})m_n >0\),
and therefore
\[
0<
\widetilde\gamma_{n+1}
=
\frac{\gamma_{n+1}}{m_{n+1}}
<1.
\]
By assumption,
\(\sum_{n\geq0}(1-m_n)=(1-m_0)\sum_{n\geq0}r_n<\infty\).
Since \(m_n\geq m_0>0\), it follows that
\(\prod_{n\geq0}m_n>0\) almost surely.

For \(x\in C\), define a probability measure \(R_n(x,\cdot)\) on \(S\) as
follows. If \(\mu_n(S\setminus C)>0\), set
\[
R_n(x,i):=
P(x,i)\mathbf 1_{ \{i\in C\} }
+
q(x)\frac{\mu_n(i)}{\mu_n(S\setminus C)}\mathbf 1_{ \{i\in S\setminus C\}},
\qquad i\in S.
\]
Since \(C\) is closed, for \(x\in C\),
\[
\sum_{i\in C}P(x,i)=1-q(x),
\]
and hence \(R_n(x,\cdot)\) is a probability measure. If \(\mu_n(S\setminus C)=0\), set
\(R_n(x,\cdot):=\delta_{x_C}\), where \(x_C\in C\) is fixed. For \(x\notin C\), extend \(R_n(x,\cdot)\) arbitrarily, for instance by \(R_n(x,\cdot)=\delta_{x_C}\).

\medskip
We now define the conditional law of \((X_{n+1},\widetilde X_{n+1},B_{n+1})\) given \(\mathcal F_n\).  On the event \(\{B_n=1\}\), for \(i\in S\) and \(j\in C\), we set
\begin{align}
\mathbb P(X_{n+1}=i, \ \widetilde X_{n+1}=j,\ B_{n+1}=1 \mid  \mathcal F_n)&=
m_n\,K_{\widetilde\mu_n}(\widetilde X_n,j)\mathbf 1_{ \{i=j\} },\label{eq:coupling-B1-one}\\
\mathbb P(X_{n+1}=i,\ \widetilde X_{n+1}=j,\ B_{n+1}=0\mid\mathcal F_n)
&=
(1-m_n)\,K_{\widetilde{\mu}_n}(\widetilde X_n,j)\,R_n(X_n,i).\label{eq:coupling-B1-zero}
\end{align}

On the event \(\{B_n=0\}\), for \(i\in S\) and \(j\in C\), we set
\begin{align}
\mathbb P(X_{n+1}=i,\ \widetilde X_{n+1}=j,\ B_{n+1}=0\mid\mathcal F_n)
&=
K_{\mu_n}(X_n,i)\,K_{\widetilde\mu_n}(\widetilde X_n,j),\label{eq:coupling-B0-zero}\\
\mathbb P(X_{n+1}=i,\ \widetilde X_{n+1}=j,\ B_{n+1}=1\mid\mathcal F_n)
&=0.\label{eq:coupling-B0-one}
\end{align}
Finally, \(\mu_{n+1}\) and \(\widetilde\mu_{n+1}\) are defined by (i) and (ii). This recursively defines the law of the whole process.

We claim that, for every \(n\geq 0\),
\begin{equation}
\label{eq:coupling-claim}
\widetilde\mu_n\in\Delta^C
\qquad
\text{and on} \ \{B_n=1\}: \quad
X_n=\widetilde X_n\in C,
\quad
\mu_n(\cdot\cap C)=m_n\widetilde\mu_n(\cdot).
\end{equation}
The claim holds at \(n=0\) by the definitions of
\(\widetilde X_0\) and \(\widetilde\mu_0\). Assume that it holds at time
\(n\). Since \(\widetilde\mu_n\in\Delta^C\) and \(C\) is closed, \(K_{\widetilde\mu_n}(\widetilde X_n,\cdot)\) is supported on \(C\). Thus \(\widetilde X_{n+1}\in C\) and 
\(\widetilde\mu_{n+1}\in\Delta^C\) almost surely by construction.

Moreover,
\(\{B_{n+1}=1\}\subset\{B_n=1\}\), and
\eqref{eq:coupling-B1-one} is supported on the diagonal. Hence, on
\(\{B_{n+1}=1\}\),
\[
X_{n+1}=\widetilde X_{n+1}\in C.
\]
Using
\(m_{n+1}-\gamma_{n+1}=(1-\gamma_{n+1})m_n\), we obtain
\[
m_{n+1}\widetilde\mu_{n+1}
=
(1-\gamma_{n+1})m_n\widetilde\mu_n
+
\gamma_{n+1}\delta_{\widetilde X_{n+1}}
=
\mu_{n+1}(\cdot\cap C),
\] which proves \eqref{eq:coupling-claim}.

We now verify the properties in the statement. By \eqref{eq:coupling-claim}, on \(\{B_n=1\}\),
\[
\mu_n(C)=m_n
\qquad\text{and}\qquad
\mu_n(S\setminus C)=1-m_n.
\]
We show that
\begin{equation}
\label{eq:K-decomposition}
K_{\mu_n}(X_n,\cdot)
=
m_nK_{\widetilde\mu_n}(\widetilde X_n,\cdot)
+
(1-m_n)R_n(X_n,\cdot).
\end{equation}
If \(m_n=1\), then
\(\mu_n=\widetilde\mu_n\) and \(X_n=\widetilde X_n\), so the identity is
immediate. Suppose \(m_n<1\). On \(C\),
\(R_n(X_n,i)=P(X_n,i)\) and
\(\mu_n(i)=m_n\widetilde\mu_n(i)\). On \(S\setminus C\),
both \(P(X_n,\cdot)\) and
\(K_{\widetilde\mu_n}(\widetilde X_n,\cdot)\) vanish, while
\[
(1-m_n)R_n(X_n,i)=q(X_n)\mu_n(i).
\]
This proves \eqref{eq:K-decomposition}.

On \(\{B_n=0\}\), the two required marginals follow directly from
\eqref{eq:coupling-B0-zero}. On \(\{B_n=1\}\), summing
\eqref{eq:coupling-B1-one}--\eqref{eq:coupling-B1-zero} and using
\eqref{eq:K-decomposition} gives the \(K_{\mu_n}\)-marginal for
\(X_{n+1}\); summing only over \(X_{n+1}\) and \(B_{n+1}\) gives the
\(K_{\widetilde\mu_n}\)-marginal for \(\widetilde X_{n+1}\).

It remains to prove (iii). Let \(Q_n\) be the transition kernel on \(\{0,1\}\) defined by
\[
Q_n=
\begin{pmatrix}
1 & 0\\
1-m_n & m_n
\end{pmatrix}.
\]
For \(j\in C\) and \(b\in\{0,1\}\), summing the transition rules over
\(X_{n+1}\) gives
\[
\mathbb P(\widetilde X_{n+1}=j,B_{n+1}=b\mid\mathcal F_n)
=
K_{\widetilde\mu_n}(\widetilde X_n,j)Q_n(B_n,b).
\]
Since \(\widetilde\mu_n\) is determined by \(\mathcal H\) and
\((\widetilde X_0,\ldots,\widetilde X_n)\), induction shows that, given
\(\mathcal H\), the processes \((B_n)\) and \((\widetilde X_n)\) are
independent and that \(B\) has transition kernels \(Q_n\).

Since \(B_0=1\),
\[
\mathbb P(B_k=1,\ 0\leq k\leq N\mid\mathcal H)
=
\prod_{n=0}^{N-1}m_n.
\]
Letting \(N\to\infty\) proves the formula for
\(\mathbb P(\mathcal E_C\mid\mathcal H)\).

Finally, (iv) is \eqref{eq:coupling-claim}. On \(\mathcal E_C\), the identity
\(\mu_n(\cdot\cap C)=m_n\widetilde\mu_n(\cdot)\) and the fact that
\(\widetilde\mu_n\in\Delta^C\) imply
\[
\|\mu_n-\widetilde\mu_n\|_1
=
\sum_{i\in C}(1-m_n)\widetilde\mu_n(i)
+
\sum_{i\in S\setminus C}\mu_n(i)
=
2(1-m_n).
\]
Since \(m_n\uparrow1\), the right-hand side converges to \(0\).
\end{proof}

The same construction applies when the step-size sequence is
\(\mathcal H\)-measurable: conditionally on \(\mathcal H\), the sequence is fixed and the preceding proof is unchanged. We use this form below with
\(\mathcal H=\mathcal F_{\tau_\alpha}\).

\subsubsection{Selection associated with a maximal class}

We now fix \(\alpha\in\{1,\ldots,\kappa\}\) and set
\[
C_\alpha:=\overline{S_\alpha}.
\]
The coupling lemma will be applied after the process has reached a point of
\(S_\alpha\), and hence has entered \(C_\alpha\). We therefore first record a
simple accessibility consequence of Hypothesis~\ref{hyp:initial-support}: every
prescribed state of \(S\), and in particular every point of \(S_\alpha\), can be
visited with positive probability.
\begin{lemma}[Positive-probability visit to a prescribed state]
\label{lem:visit-accessible}
For every \(j\in S\),
\[
\mathbb P(\exists n\geq1:\ X_n=j)>0.
\]
\end{lemma}

\begin{proof}
Fix \(j\in S\). By Hypothesis~\ref{hyp:initial-support},
\[
S=\overline{\supp(\mu_0)}
\qquad\text{almost surely}.
\]
The events
\[
E_i:=\{\mu_0(i)>0\},
\qquad i\rightsquigarrow j,
\]
cover an almost sure event. Since \(S\) is finite, one of them has positive probability; fix such an \(i\).

Choose a \(P\)-path
\[
i=z_0,\ldots,z_\ell=j
\]
with
\[
\lambda_{ij}:=\prod_{r=0}^{\ell-1}P(z_r,z_{r+1})>0.
\]
% At the terminal point of this path,
% \[
% K_{\mu_{\ell_{X_0}}}
% \bigl(x_{\ell_{X_0}}^{X_0},i\bigr)
% \geq
% q\bigl(x_{\ell_{X_0}}^{X_0}\bigr)
% \mu_{\ell_{X_0}}(i).
% \]

For every \(x\in S\), choose a path
\[
x=x_0^x,\ldots,x_{\ell_x}^x
\]
such that
\[
d_x
:=
\left(\prod_{r=0}^{\ell_x-1}
P(x_r^x,x_{r+1}^x)\right)q(x_{\ell_x}^x)>0.
\]
By finiteness of \(S\), let
\[
p:=1+\max_x\ell_x,
\qquad
d:=\min_x d_x>0.
\]
Starting from \(X_0\), the process may follow the corresponding path, jump to
\(i\), and then follow \(z_0,\ldots,z_\ell\). Since
\(K_\mu\geq P\) and
\[
\mu_{\ell_{X_0}}(i)\geq r_p\mu_0(i),
\]
the probability of this event is bounded below on \(E_i\) by
\(d\,r_p\mu_0(i)\lambda_{ij}\). Hence
\[
\mathbb P(\exists n\geq1:X_n=j)
\geq
d\,r_p\lambda_{ij}
\mathbb E[\mathbf1_{E_i}\mu_0(i)]
>0.
\]

\end{proof}

The previous lemma allows us to hit a prescribed point
\(j_\alpha\in S_\alpha\) with positive probability. Once this has happened, we
restart the construction at the hitting time of \(j_\alpha\) and apply the
coupling on the closed set \(C_\alpha\). It remains to identify the QSD selected
by the reduced dynamics on \(C_\alpha\). The next lemma shows that, for the
restricted kernel on \(C_\alpha\), the class \(S_\alpha\) has the largest
Perron value among the maximal classes, and that the associated QSD is still
\(\nu_\alpha\).

\begin{lemma}
\label{lem:Calpha-top}
Fix \(\alpha\in\{1,\ldots,\kappa\}\) and set
\[
C_\alpha:=\overline{S_\alpha},
\qquad
P_{C_\alpha}:=P|_{C_\alpha\times C_\alpha}.
\]
Then \(C_\alpha\) is closed. The restricted kernel \(P_{C_\alpha}\) satisfies
Hypotheses~\ref{hyp:main} and~\ref{hyp:main-two}. Moreover, \(S_\alpha\) is a maximal class for \(P_{C_\alpha}\), and its Perron value is strictly larger than that of every other communicating class of \(P_{C_\alpha}\).

The QSD associated with \(S_\alpha\) for the restricted kernel is \(\nu_\alpha\).
\end{lemma}

\begin{proof}
The set \(C_\alpha=\overline{S_\alpha}\) is closed by transitivity of the accessibility relation. The communicating classes of \(P_{C_\alpha}\) are those classes of \(P\)
contained in \(C_\alpha\). Since \(C_\alpha\) is closed, every class accessible from a class contained in
\(C_\alpha\) is also contained in \(C_\alpha\). A class contained in
\(C_\alpha\) is therefore maximal for the restricted kernel if and only if it
is maximal for \(P\). Hence Hypothesis~\ref{hyp:main-two} is inherited.
Hypothesis~\ref{hyp:main} is inherited because a path from a state of
\(C_\alpha\) to the cemetery state cannot leave the closed set
\(C_\alpha\) before killing.

If \(S_\beta\subset C_\alpha\), then
\(S_\alpha\rightsquigarrow S_\beta\), or equivalently
\(S_\beta\preccurlyeq S_\alpha\). Hence, for
\(\beta\neq\alpha\), maximality of \(S_\alpha\) gives
\(\rho_\beta<\rho_\alpha\). Finally,
\(\supp(\nu_\alpha)=C_\alpha\) and
\[
\nu_\alpha P_{C_\alpha}
=
\rho_\alpha\nu_\alpha.
\]
Proposition~\ref{prop:setqsd}, applied to \(P_{C_\alpha}\), therefore
identifies \(\nu_\alpha\) as the QSD associated with \(S_\alpha\).
\end{proof}

\begin{proposition}[Selection associated with a maximal class]
\label{prop:strong-one-alpha}

Fix \(\alpha\in\{1,\ldots,\kappa\}\) and set
\[
C_\alpha:=\overline{S_\alpha}.
\]
Then
\[
\mathbb P\Bigl(\bigl\{X_n\in C_\alpha \text{ for all sufficiently large }n\bigr\}
\cap
\bigl\{\mu_n\to\nu_\alpha\bigr\}\Bigr)>0.
\]
\end{proposition}

\begin{proof}
Choose \(j_\alpha\in S_\alpha\), and let
\[
\tau_\alpha:=\inf\{n\geq 1:\ X_n=j_\alpha\}.
\]
By Lemma~\ref{lem:visit-accessible},
\(\mathbb P(\tau_\alpha<\infty)>0\). On this event,
\(X_{\tau_\alpha}=j_\alpha\in C_\alpha\) and
\[
\mu_{\tau_\alpha}(C_\alpha)
\geq
\mu_{\tau_\alpha}(j_\alpha)
\geq
\gamma_{\tau_\alpha}
>0.
\]
On \(\{\tau_\alpha<\infty\}\), we apply the conditional coupling lemma with
initial sigma-field \(\mathcal H=\mathcal F_{\tau_\alpha}\). For this purpose,
consider the shifted process
\[
X_n^{(\tau)}:=X_{\tau_\alpha+n},
\qquad
\mu_n^{(\tau)}:=\mu_{\tau_\alpha+n},
\qquad n\geq0.
\]
Its tail step-size sequence is
\[
\gamma_{n+1}^{(\tau)}
=
\gamma_{\tau_\alpha+n+1},
\]
and the associated product sequence is
\[
r_0^{(\tau)}:=1,
\qquad
r_n^{(\tau)}
:=
\prod_{k=1}^n
\bigl(1-\gamma_{\tau_\alpha+k}\bigr)
=
\frac{r_{\tau_\alpha+n}}{r_{\tau_\alpha}},
\qquad n\geq1.
\]
Thus
\[
\sum_{n\geq0}r_n^{(\tau)}<\infty
\qquad\text{on } \ \{\tau_\alpha<\infty\}.
\]

We may therefore apply the coupling Lemma~\ref{lem:coupling-reduced}, conditionally on
\(\mathcal F_{\tau_\alpha}\), with \(C=C_\alpha\), initial state
\(X_0^{(\tau)}=j_\alpha\), initial measure
\(\mu_0^{(\tau)}=\mu_{\tau_\alpha}\), and tail step-size sequence
\((\gamma_n^{(\tau)})\). Let
\((\widetilde X_n,\widetilde\mu_n,B_n)\)
be the reduced process and coupling process supplied by
Lemma~\ref{lem:coupling-reduced}. Let
\[
\mathcal E_\alpha^{\mathrm{coup}}
:=
\{B_n=1\ \forall n\geq0\}
\]
be the associated good coupling event.
By Lemma~\ref{lem:coupling-reduced}, on \(\{\tau_\alpha<\infty\}\),
\[
\mathbb P\left(
\mathcal E_\alpha^{\mathrm{coup}}
\,\middle|\,
\mathcal F_{\tau_\alpha}
\right)>0.
\]
Moreover, on \(\mathcal E_\alpha^{\mathrm{coup}}\),
\[
X_{\tau_\alpha+n}=\widetilde X_n\in C_\alpha
\qquad \forall n\geq0,
\]
and
\[
\|\mu_{\tau_\alpha+n}-\widetilde\mu_n\|_1\longrightarrow0.
\]

We now look at the reduced process alone. 
Since \(\widetilde\mu_0(j_\alpha)>0\) and
\(\overline{\{j_\alpha\}}=C_\alpha\), the reduced initial condition satisfies
\[
\widetilde X_0=j_\alpha\in\supp(\widetilde\mu_0),
\qquad
\overline{\supp(\widetilde\mu_0)}=C_\alpha.
\]

% Since
% \[
% \widetilde\mu_0=\frac{\mu_{\tau_\alpha}(\cdot\cap C_\alpha)}{\mu_{\tau_\alpha}(C_\alpha)},
% \]
% we have \(\widetilde\mu_0(j_\alpha)>0\), so
% \(
% j_\alpha\in \supp(\widetilde\mu_0).
% \)
% Because \(j_\alpha\in S_\alpha\) and all points of \(S_\alpha\) communicate,
% the set of states accessible from \(j_\alpha\) is
% \[
% \overline{\{j_\alpha\}}=\overline{S_\alpha}=C_\alpha.
% \]
% Since also \(\supp(\widetilde\mu_0)\subset C_\alpha\), it follows that
% \[
% \overline{\supp(\widetilde\mu_0)}=C_\alpha.
% \]

By Lemma~\ref{lem:Calpha-top}, the restricted chain on \(C_\alpha\) satisfies
Hypotheses~\ref{hyp:main} and~\ref{hyp:main-two}, its maximal class with largest
Perron value is \(S_\alpha\), and the corresponding QSD is \(\nu_\alpha\).

The reduced step sizes are
\[
\widetilde\gamma_{n+1}
=
\frac{\gamma_{\tau_\alpha+n+1}}{m_{n+1}^{(\tau)}},
\qquad
m_n^{(\tau)}
=
1-r_n^{(\tau)}
\bigl(1-m_0^{(\tau)}\bigr),
\]
where
\[
m_0^{(\tau)}
=
\mu_{\tau_\alpha}(C_\alpha).
\]
Since
\[
\gamma_{\tau_\alpha+n+1}
\leq
\widetilde\gamma_{n+1}
\leq
\frac{\gamma_{\tau_\alpha+n+1}}{m_0^{(\tau)}},
\]
the sum of the reduced step sizes diverges and
\(\widetilde\gamma_n\log n\to0\). Moreover,
\((\gamma_{\tau_\alpha+n})\) is eventually non-increasing while
\((m_n^{(\tau)})\) is non-decreasing, so
\((\widetilde\gamma_n)\) is eventually non-increasing.

Applying Theorem~\ref{th:main} to each conditional realization of the reduced
initial data gives
\[
\mathbb P\left(
\widetilde\mu_n\to\nu_\alpha
\,\middle|\,
\mathcal F_{\tau_\alpha}
\right)>0
\qquad
\text{on }\{\tau_\alpha<\infty\}.
\]

Conditionally on \(\mathcal F_{\tau_\alpha}\), the good coupling event is
independent of the reduced trajectory, and therefore of
\(\{\widetilde\mu_n\to\nu_\alpha\}\). Hence their intersection has positive
conditional probability on \(\{\tau_\alpha<\infty\}\).

On this intersection,
\[
X_{\tau_\alpha+n}\in C_\alpha
\quad\forall n\geq0,
\qquad
\|\mu_{\tau_\alpha+n}-\widetilde\mu_n\|_1\to0,
\]
and hence \(\mu_{\tau_\alpha+n}\to\nu_\alpha\). Taking expectations on \(\{\tau_\alpha<\infty\}\) yields the desired positive unconditional probability.

\end{proof}

\subsubsection{Explicit trapping estimate in closed sets}

The coupling argument above proves the positive-probability selection statement.
We now record a direct estimate of the trapping mechanism in a closed set. This
estimate is useful independently of the coupling construction and gives the
explicit lower bound appearing in Theorem~\ref{thm:strong-regime}.

\begin{proposition}[Trapping in a closed set]
\label{prop:closed-trapping}
Assume the strong reinforcement condition \eqref{eq:sum-rn-finite} holds. Let
\(C\subset S\) be non-empty and closed, and let \(\tau\) be a finite stopping
time. Then, on \(\{X_\tau\in C\}\),
\[
\mathbb P\left(
X_n\in C\ \forall n\geq \tau
\,\middle|\,
\mathcal F_\tau
\right)>0
\qquad\text{almost surely }.
\]
More precisely, for every deterministic \(N\geq0\), on \(\{X_N\in C\}\),
\[
\mathbb P\left(
X_n\in C\ \forall n\geq N
\,\middle|\,
\mathcal F_N
\right)
\geq
\prod_{n\geq N}
\left(
1-\mu_N(S\setminus C)\frac{r_n}{r_N}
\right)>0.
\]
\end{proposition}

\begin{proof}
We first prove the estimate for deterministic \(N\). The stopping-time statement will then follow by
decomposing over the events \(\{\tau=N\}\).

Fix \(N\geq0\), and for \(n\geq N\), set
\[
A_n:=\{X_k\in C,\ \forall k=N,\ldots,n\}.
\]
We work on the event \(\{X_N\in C\}\), so that \(A_N\) holds. Since \(C\) is
closed, \(P(x,S\setminus C)=0\) for every \(x\in C\). Hence, on \(A_n\),
\[
\begin{aligned}
\mathbb P(X_{n+1}\notin C\mid\mathcal F_n)
&=
K_{\mu_n}(X_n,S\setminus C)\\
&=
P(X_n,S\setminus C)+q(X_n)\mu_n(S\setminus C)\\
&=
q(X_n)\mu_n(S\setminus C)
\leq
\mu_n(S\setminus C).
\end{aligned}
\]
Therefore, on \(A_n\),
\[
\mathbb P(X_{n+1}\in C\mid\mathcal F_n)
\geq
1-\mu_n(S\setminus C).
\]

Moreover, on \(A_n\), no point outside \(C\) is visited after time \(N\). Hence,
for \(N\leq k<n\),
\[
\mu_{k+1}(S\setminus C)
=
(1-\gamma_{k+1})\mu_k(S\setminus C).
\]
Iterating from \(N\) to \(n\) gives
\[
\mu_n(S\setminus C)
=
\mu_N(S\setminus C)\prod_{k=N+1}^n(1-\gamma_k)
=
\mu_N(S\setminus C)\frac{r_n}{r_N}.
\]
Thus, on \(A_n\),
\[
\mathbb P(X_{n+1}\in C\mid\mathcal F_n)
\geq
1-\mu_N(S\setminus C)\frac{r_n}{r_N}.
\]

Using
\[
\mathbb P(A_{n+1}\mid\mathcal F_N)
=
\mathbb E\left[
\mathbf 1_{A_n}
\mathbb P(X_{n+1}\in C\mid\mathcal F_n)
\,\middle|\,
\mathcal F_N
\right],
\]
we obtain
\[
\mathbb P(A_{n+1}\mid\mathcal F_N)
\geq
\left(
1-\mu_N(S\setminus C)\frac{r_n}{r_N}
\right)
\mathbb P(A_n\mid\mathcal F_N).
\]
Iterating this inequality gives, for every \(m\geq N\),
\[
\mathbb P(A_m\mid\mathcal F_N)
\geq
\prod_{n=N}^{m-1}
\left(
1-\mu_N(S\setminus C)\frac{r_n}{r_N}
\right).
\]
Since \(A_m\downarrow\{X_n\in C\ \forall n\geq N\}\) as \(m\to\infty\), we get
\[
\mathbb P(X_n\in C\ \forall n\geq N\mid\mathcal F_N)
\geq
\prod_{n\geq N}
\left(
1-\mu_N(S\setminus C)\frac{r_n}{r_N}
\right).
\]

It remains to check that the infinite product is strictly positive. On \(\{X_N\in C\}\), one has \(\mu_N(C)>0\): this follows from
Hypothesis~\ref{hyp:initial-support} if \(N=0\), and from
\(\mu_N(X_N)\geq\gamma_N\) if \(N\geq1\). Hence
\(\mu_N(S\setminus C)<1\). Since \(r_n/r_N\leq1\) and
\[
\sum_{n\geq N}
\mu_N(S\setminus C)\frac{r_n}{r_N}<\infty,
\]
the infinite product is strictly positive. This proves the deterministic statement.

For a finite stopping time \(\tau\),
\[
\mathbb P\left(
X_n\in C\ \forall n\geq\tau
\,\middle|\,
\mathcal F_\tau
\right)
=
\sum_{N\geq0}
\mathbf1_{\{\tau=N\}}
\mathbb P\left(
X_n\in C\ \forall n\geq N
\,\middle|\,
\mathcal F_N
\right),
\]
which is strictly positive on \(\{X_\tau\in C\}\).

\end{proof}

\begin{proof}[Proof of Theorem~\ref{thm:strong-regime}]
The positive-probability selection statement follows from
Proposition~\ref{prop:strong-one-alpha}. Applying
Proposition~\ref{prop:closed-trapping} with
\(C=\overline{S_\alpha}\) gives
\eqref{eq:closed-trapping-lb}.
\end{proof}

\textbf{Acknowledgement.} This work was supported in part by the Swiss National Science Foundation (FNS Grant 200020 196999).

\appendix

\section{Genericity of Hypothesis~\ref{hyp:main-two}}
\label{app:genericity}

Let \(\Gamma=(V,E)\) be the incidence graph of the kernel \(P\) fixed
throughout the paper. To shorten the notation, we write
\[
i\to j
\quad\Longleftrightarrow\quad
(i,j)\in E.
\]

For the purpose of this genericity statement, we assume that the fixed graph
\(\Gamma\) has no singleton communicating class without a self-loop, that is,
\[
S_\alpha=\{i\}
\quad\Longrightarrow\quad
i\to i.
\]
Indeed, if \(S_\alpha=\{i\}\) and \(i\not\to i\), then the corresponding
block is the one-by-one zero matrix for every kernel with incidence graph
\(\Gamma\), and its Perron value is therefore identically equal to zero on
the whole incidence stratum. We denote by \(\mathcal M_\Gamma\) the set of sub-Markovian kernels having
exactly \(\Gamma\) as incidence graph. Thus \(Q\in\mathcal M_\Gamma\) if and
only if
\[
i\to j
\quad\Longleftrightarrow\quad
Q(i,j)>0,
\qquad
i\to\partial
\quad\Longleftrightarrow\quad
q_Q(i)>0,
\]
where
\(
q_Q(i):=1-\sum_{j\in S}Q(i,j).
\)

\begin{proposition}
There exists a relatively open and dense subset of
\(\mathcal M_\Gamma\), of full relative Lebesgue measure, on which the
Perron values of all communicating classes are pairwise distinct. In
particular, Hypothesis~\ref{hyp:main-two} holds on this subset.
\end{proposition}

\begin{proof}
When nonempty, \(\mathcal M_\Gamma\) is a connected, relatively open and
convex subset of its natural affine hull. Indeed, row by row, the positive
transition probabilities prescribed by \(\Gamma\), together with the killing
probability when \(i\to\partial\), form the relative interior of a simplex.
Relative Lebesgue measure below refers to Lebesgue measure on this affine
hull.

All kernels in \(\mathcal M_\Gamma\) have the same communicating classes
\(S_1,\ldots,S_L\). Moreover, since accessibility depends only on the
incidence graph, every kernel in \(\mathcal M_\Gamma\) satisfies
Hypothesis~\ref{hyp:main}.

For \(Q\in\mathcal M_\Gamma\), let \(\rho_\alpha(Q)\) denote the Perron
value of the irreducible block \(Q_{S_\alpha}\). Since this Perron
eigenvalue is algebraically simple, the map
\(
Q\longmapsto\rho_\alpha(Q)
\)
is real-analytic on \(\mathcal M_\Gamma\).

Fix two distinct communicating classes \(S_\alpha\) and \(S_\beta\), and
define
\[
f_{\alpha,\beta}(Q)
:=
\rho_\alpha(Q)-\rho_\beta(Q).
\]
The function \(f_{\alpha,\beta}\) is real-analytic. We now show that it is
not identically zero.

Fix \(Q\in\mathcal M_\Gamma\). Since the cemetery state is accessible from
\(S_\alpha\), choose a path from \(S_\alpha\) to \(\partial\), and let
\(x\in S_\alpha\) be the last state of \(S_\alpha\) visited by this path.
Then either
\(
Q(x,z)>0
\qquad\text{for some }z\notin S_\alpha,
\)
or
\(
q_Q(x)>0.
\)
There also exists \(y\in S_\alpha\) such that \(Q(x,y)>0\). If
\(|S_\alpha|\geq2\), this follows from the irreducibility of
\(Q_{S_\alpha}\); if \(S_\alpha=\{x\}\), it follows from the assumed
self-loop condition.

In the first case, choose
\(
0<\varepsilon<Q(x,z)
\)
and define \(Q_\varepsilon\) by
\[
Q_\varepsilon(x,y)=Q(x,y)+\varepsilon,
\qquad
Q_\varepsilon(x,z)=Q(x,z)-\varepsilon.
\]
In the second case, choose
\(
0<\varepsilon<q_Q(x)
\)
and define \(Q_\varepsilon\) by
\[
Q_\varepsilon(x,y)=Q(x,y)+\varepsilon;
\]
then
\(
q_{Q_\varepsilon}(x)=q_Q(x)-\varepsilon>0.
\)
In both cases, every other transition probability is left unchanged.

The coordinate that is decreased remains strictly positive, while the
coordinate \(Q(x,y)\) was already strictly positive. Therefore the incidence
graph is unchanged and
\(
Q_\varepsilon\in\mathcal M_\Gamma.
\)
Moreover, the only diagonal block that is modified is the block associated
with \(S_\alpha\).

Write
\[
A:=Q_{S_\alpha},
\qquad
A_\varepsilon:=(Q_\varepsilon)_{S_\alpha}.
\]
Let \(r>0\) be a right Perron eigenvector of \(A\), and let
\(\ell_\varepsilon>0\) be a left Perron eigenvector of
\(A_\varepsilon\). Since the only modification inside the block is the
increase of its \((x,y)\)-entry,
\[
A_\varepsilon r
=
\rho_\alpha(Q)r+\varepsilon r(y)e_x,
\]
where \(e_x\) denotes the coordinate vector corresponding to \(x\).
Consequently,
\begin{align*}
\rho_\alpha(Q_\varepsilon)\,\ell_\varepsilon r
&=
\ell_\varepsilon A_\varepsilon r\\
&=
\rho_\alpha(Q)\,\ell_\varepsilon r
+
\varepsilon\ell_\varepsilon(x)r(y).
\end{align*}
Since
\[
\ell_\varepsilon(x)>0,
\qquad
r(y)>0,
\qquad
\ell_\varepsilon r>0,
\]
we obtain
\[
\rho_\alpha(Q_\varepsilon)>\rho_\alpha(Q).
\]

The block associated with \(S_\beta\) has not been modified, and hence
\[
\rho_\beta(Q_\varepsilon)=\rho_\beta(Q).
\]
Therefore
\[
f_{\alpha,\beta}(Q_\varepsilon)
>
f_{\alpha,\beta}(Q),
\]
so \(f_{\alpha,\beta}\) is not constant and, in particular, is not
identically zero. It follows from the zero-set theorem for real-analytic functions that
\[
\left\{
Q\in\mathcal M_\Gamma:
\rho_\alpha(Q)=\rho_\beta(Q)
\right\}
\]
is relatively closed, has empty relative interior, and has zero relative
Lebesgue measure.

There are only finitely many pairs of communicating classes. Consequently,
the exceptional set
\[
\bigcup_{1\leq\alpha<\beta\leq L}
\left\{
Q\in\mathcal M_\Gamma:
\rho_\alpha(Q)=\rho_\beta(Q)
\right\}
\]
is relatively closed, has empty relative interior, and has zero relative
Lebesgue measure.

Its complement is therefore relatively open and dense in
\(\mathcal M_\Gamma\), and has full relative Lebesgue measure. On this
complement, the Perron values of all communicating classes are pairwise
distinct. In particular, the Perron values associated with the maximal
classes are pairwise distinct, which is
Hypothesis~\ref{hyp:main-two}.
\end{proof}

\bibliographystyle{abbrv}
\bibliography{QSDreducible}

@article{BH96,
title={Asymptotic pseudotrajectories and chain recurrent flows, with applications},
author={ M.~Bena\"{\i}m and M.~W. Hirsch},
journal={Journal of Dynamics and Differential Equations},
volume={8},
pages={141-176},
year={1996}
}

@article{pemantle1992vertex,
  title={Vertex-reinforced random walk},
  author={Pemantle, Robin},
  journal={Probability Theory and Related Fields},
  volume={92},
  number={1},
  pages={117--136},
  year={1992},
  publisher={Springer}
}

@article {B97,
    AUTHOR = {Bena{\"{\i}}m, Michel},
     TITLE = {Vertex-reinforced random walks and a conjecture of {P}emantle},
   JOURNAL = {The Annals of Probability},
  FJOURNAL = {The Annals of Probability},
    VOLUME = {25},
      YEAR = {1997},
    NUMBER = {1},
     PAGES = {361--392},
      ISSN = {0091-1798},
     CODEN = {APBYAE},
   MRCLASS = {60J10 (34F05 60J15)},
  MRNUMBER = {1428513 (98d:60128)},
MRREVIEWER = {J{\'a}nos T{\'o}th},
       DOI = {10.1214/aop/1024404292},
       URL = {http://dx.doi.org/10.1214/aop/1024404292},
}

@Article{AFP,
author = {Aldous, D and Flannery, B and Palacios, J-L},
title = {Two applications of urn processes: the fringe analysis of search
  trees and the simulation of quasi-stationary distributions of Markov chains},
journal = {Probability in the Engineering and Informational Sciences},
year = {1988},
OPTkey = {â¢},
volume = {2},
number = {3},
pages = {293--307},
OPTmonth = {â¢},
OPTnote = {â¢},
OPTannote = {â¢}
}

@article {BC2015,
    AUTHOR = {Bena\"{i}m, Michel and Cloez, Bertrand},
     TITLE = {A stochastic approximation approach to quasi-stationary
              distributions on finite spaces},
   JOURNAL = {Electronic Communications in Probability},
  FJOURNAL = {Electronic Communications in Probability},
    VOLUME = {20},
      YEAR = {2015},
     PAGES = {no. 37, 14},
      ISSN = {1083-589X},
   MRCLASS = {60J10 (60B12 60J20 65C50)},
  MRNUMBER = {3352332},
MRREVIEWER = {Nizar Demni},
       DOI = {10.1214/ECP.v20-3956},
       URL = {https://doi.org/10.1214/ECP.v20-3956},
}

@article{bansaye2022non,
  title={A non-conservative Harris ergodic theorem},
  author={Bansaye, Vincent and Cloez, Bertrand and Gabriel, Pierre and Marguet, Aline},
  journal={Journal of the London Mathematical Society},
  volume={106},
  number={3},
  pages={2459--2510},
  year={2022},
  publisher={Wiley Online Library}
}

@article{champagnat2022quasi,
  title={Quasi-stationary distributions in reducible state spaces},
  author={Champagnat, Nicolas and Villemonais, Denis},
  journal={Advances in Applied Probability},
  pages={1--37},
  year={2022},
  publisher={Cambridge University Press}
}

@article{benaim2018stochastic,
  title={Stochastic approximation of quasi-stationary distributions on compact spaces and applications},
  author={Benaim, Michel and Cloez, Bertrand and Panloup, Fabien},
  journal={The Annals of Applied Probability},
  volume={28},
  number={4},
  pages={2370--2416},
  year={2018},
  publisher={JSTOR}
}

@article{pollett2008quasi,
  title={Quasi-stationary distributions for reducible absorbing Markov chains in discrete time},
  author={Pollett, PK and van Doorn, Erik A},
  journal={Markov processes and related fields},
  volume={15},
  pages={191--204},
  year={2009}
}

@book{conley1978isolated,
  title={Isolated invariant sets and the Morse index},
  author={Conley, Charles C},
  volume={38},
  year={1978},
  publisher={American Mathematical Soc.}
}

@incollection {B99,
    AUTHOR = {Bena{\"{\i}}m, M.},
     TITLE = {Dynamics of stochastic approximation algorithms},
 BOOKTITLE = {S\'eminaire de {P}robabilit\'es, {XXXIII}},
    SERIES = {Lecture Notes in Math.},
    VOLUME = {1709},
     PAGES = {1--68},
 PUBLISHER = {Springer, Berlin},
      YEAR = {1999},
   MRCLASS = {62L20 (37B99 37H99 60C05 60G99)},
  MRNUMBER = {1767993 (2001d:62081)},
       DOI = {10.1007/BFb0096509},
       URL = {http://dx.doi.org/10.1007/BFb0096509},
}

@inproceedings{benaim2021stochastic,
  title={Stochastic approximation of quasi-stationary distributions for diffusion processes in a bounded domain},
  author={Bena{\"\i}m, Michel and Champagnat, Nicolas and Villemonais, Denis},
  booktitle={Annales de l'Institut Henri Poincar{\'e} (B) Probabilit{\'e}s et Statistiques},
  volume={57},
  pages={726--739},
  year={2021}
}

@article{MaillerVillemonais2020,
author = {C{\'e}cile Mailler and Denis Villemonais},
title = {{Stochastic approximation on noncompact measure spaces and application to measure-valued PÃ³lya processes}},
volume = {30},
journal = {The Annals of Applied Probability},
number = {5},
publisher = {Institute of Mathematical Statistics},
pages = {2393 -- 2438},
year = {2020},
doi = {10.1214/20-AAP1561},
URL = {https://doi.org/10.1214/20-AAP1561}
}

@article {Ben96,
    AUTHOR = {Benaim, Michel},
     TITLE = {A dynamical system approach to stochastic approximations},
   JOURNAL = {SIAM Journal on Control and Optimization},
  FJOURNAL = {SIAM Journal on Control and Optimization},
    VOLUME = {34},
      YEAR = {1996},
    NUMBER = {2},
     PAGES = {437--472},
      ISSN = {0363-0129},
   MRCLASS = {62L20 (93E25)},
  MRNUMBER = {1377706},
MRREVIEWER = {K.\ M.\ Ramachandran},
       DOI = {10.1137/S0363012993253534},
       URL = {https://doi.org/10.1137/S0363012993253534},
}

@article{tough2025selection,
  title={Selection principle for the Fleming--Viot particle system on the positive half-line with constant negative drift},
  author={Tough, Oliver},
  journal={Journal of the European Mathematical Society},
  year={2025}
}

@article{schneider1986influence,
  title={The influence of the marked reduced graph of a nonnegative matrix on the Jordan form and on related properties: A survey},
  author={Schneider, Hans},
  journal={Linear Algebra and its Applications},
  volume={84},
  pages={161--189},
  year={1986},
  publisher={Elsevier}
}

@article{tarres2000pieges,
  title={Pi{\`e}ges r{\'e}pulsifs},
  author={Tarr{\`e}s, Pierre},
  journal={Comptes Rendus de l'Acad{\'e}mie des Sciences-Series I-Mathematics},
  volume={330},
  number={2},
  pages={125--130},
  year={2000},
  publisher={Elsevier}
}

@article{pages2004can,
  title={When can the two-armed bandit algorithm be trusted?},
  author={Pag{\`e}s, G and Lamberton, D and Tarr{\`e}s, P},
  journal={The Annals of Applied Probability},
  volume={14},
  pages={1424--1454},
  year={2004}
}

@article{pemantle1990nonconvergence,
  title={Nonconvergence to unstable points in urn models and stochastic approximations},
  author={Pemantle, Robin},
  journal={The Annals of Probability},
  pages={698--712},
  year={1990},
  publisher={JSTOR}
}

@article{blanchet2016analysis,
  title={Analysis of a stochastic approximation algorithm for computing quasi-stationary distributions},
  author={Blanchet, Jose and Glynn, Peter and Zheng, Shuheng},
  journal={Advances in Applied Probability},
  volume={48},
  number={3},
  pages={792--811},
  year={2016},
  publisher={Cambridge University Press}
}

@phdthesis{tarres2001pieges,
  title={Pieges des algorithmes stochastiques et marches al{\'e}atoires renforc{\'e}es par sommets},
  author={Tarr{\`e}s, Pierre},
  year={2001},
  school={Cachan, Ecole normale sup{\'e}rieure}
}

@article{panloup2026asymptotically,
  title={Asymptotically unbiased approximation of the QSD of diffusion processes with a decreasing time step Euler scheme},
  author={Panloup, Fabien and Reygner, Julien},
  journal={The Annals of Applied Probability},
  volume={36},
  number={2},
  pages={1377--1415},
  year={2026},
  publisher={Institute of Mathematical Statistics}
}

\end{document}